\documentclass[aop,MSNbibl,seceqn,dvips]{arximspdf}

\doi{10.1214/12-AOP821} 
\volume{41}
\issue{4}
\pubyear{2013}
\firstpage{3026}
\lastpage{3050}

\makeatletter
\def\geqslant{\geq}

\newcommand{\rright}{\right}
\newcommand{\lleft}{\left}
\newcommand{\rrVert}{\Vert}
\newcommand{\rrvert}{\vert}
\newcommand{\llVert}{\Vert}
\newcommand{\llvert}{\vert}
\newtheorem{theorem}{Theorem}[section]
\newcommand{\Lip}{\mathrm{Lip}}
\newtheorem{condition}{Condition}
\newtheorem{corollary}[theorem]{Corollary}
\newproclaim{definition}[theorem]{Definition}
\newtheorem{lemma}[theorem]{Lemma}
\newproclaim{notation}{Notation}
\newtheorem{proposition}[theorem]{Proposition}
\newproclaim{remark}[theorem]{Remark}

\makeatother

\begin{document}
\begin{frontmatter}

\title{Integrability and tail estimates for Gaussian rough differential equations}
\runtitle{Integrability for Gaussian RDEs}

\begin{aug}
\author[A]{\fnms{Thomas} \snm{Cass}\corref{}\thanksref{t1}\ead[label=e1]{thomas.cass@imperial.ac.uk}},
\author[A]{\fnms{Christian} \snm{Litterer}\thanksref{t2}}
\and
\author[B]{\fnms{Terry} \snm{Lyons}\thanksref{t3}}
\thankstext{t1}{Supported in part by EPSRC grant EP/F029578/1 and
Christ Church, Oxford.
Part of this research was carried out during a stay of the first two authors
at the Hausdorff Institute, Bonn.}
\thankstext{t2}{Supported by a grant of the European
Research Council (ERC Grant nr. 258237) and EPSRC Grant EP/H0005500/1.}
\thankstext{t3}{Supported by EPSRC Grants EP/F029578/1,
EP/H000100/1 and by the Oxford-Man Institute for Quantitative Finance.}
\runauthor{T. Cass, C. Litterer and T. Lyons}
\affiliation{Imperial College London, Imperial College London and\break
University of Oxford}
\address[A]{T. Cass\\
C. Litterer\\
Department of Mathematics\\
Imperial College London\\
Huxley Building\\
180 Queensgate \\
London SW7 2AZ\\
United Kingdom} 
\address[B]{T. Lyons\\
Oxford-Man Institute of Quantitative Finance\\
University of Oxford\\
Walton Well Road \\
Oxford OX2 6ED\\
United Kingdom}
\end{aug}

\received{\smonth{7} \syear{2011}}
\revised{\smonth{7} \syear{2012}}

%
\begin{abstract}
We derive explicit tail-estimates for the Jacobian of the solution flow for
stochastic differential equations driven by Gaussian rough paths. In
particular, we deduce that the Jacobian has finite moments of all order
for a
wide class of Gaussian process including fractional Brownian motion
with Hurst
parameter $H>1/4$. We remark on the relevance of such estimates to a
number of
significant open problems.
\end{abstract}

%
\begin{keyword}[class=AMS]
\kwd[Primary ]{60H10}
\kwd[; secondary ]{60G15}
\end{keyword}
\begin{keyword}
\kwd{Rough path analysis}
\kwd{Gaussian processes}
\end{keyword}

\end{frontmatter}

\section{Introduction}
Gaussian processes that are not necessarily semimartingales arise in
modeling a large variety of natural phenomena. The range of their applications
reaches from fluid dynamics (e.g., randomly forced Navier--Stokes systems~\cite{HM}), the modeling of financial markets under transaction costs
\cite{G}, to the study of internet traffic through queuing models based on
fractional Brownian motion (fBm)~\cite{GO}. These applications motivate the
study of stochastic differential equations of the form%
\begin{equation}
dY_{t}=V(Y_{t})\,dX_{t},\qquad Y ( 0 )
=y_{0},\label{Int-Eq}%
\end{equation}
driven by a Gaussian process $X$. Over the past decade extensive
progress has
been made understanding the behavior of solutions to such equations. In
particular, for the case of fBm with Hurst parameter $H>1/4$, the work
of Cass
and Friz~\cite{CF} shows the existence of the density for (\ref
{Int-Eq}) under
H\"{o}rmander's condition; Hairer et al.~\cite{BH,HP} have
shown the
smoothness of this density and established ergodicity under the regime $H>1/2$.

Various recent works (Coutin and Qian~\cite{CQ}, Ledoux, Qian and Zhang
\cite{LQZ},
Friz and Victoir~\cite{FV}, Lyons and Hambly~\cite{HL}) have explored
the use of rough
paths to understand differential equations driven by nonsemimartingale noise
processes. Within this framework we can make sense of the solutions to
(\ref{Int-Eq}) driven by a broader class of Gaussian noises than classical
analysis based on Young integration would allow. This class includes
fBm with
$H>1/4$. Thus, if we consider the flow $U_{t\leftarrow0}^{\mathbf
{X}} (
y_{0} ) \equiv Y_{t}$ of the RDE (\ref{Int-Eq}), then
under sufficient regularity on $V,$ the~map $U_{t\leftarrow0}^{\mathbf
{X}%
} ( \cdot ) $ is a differentiable function (see, e.g.,
\cite{FV}), and its derivative (``the Jacobian'')
\[
J_{t\leftarrow0}^{\mathbf{X}}(y_{0})\equiv DU_{t\leftarrow0}^{\mathbf{X}
}
( \cdot ) |_{\cdot=y_{0}}%
\]
satisfies path-by-path an RDE of linear growth driven by $\mathbf{X}$.

A careful reading of the diverse applications in~\cite{BH,HP}
reveals a surprisingly generic common obstacle to the extensions of such
results to the rough path regime. This obstacle eventually boils down
to the
need for sharp estimates on the integrability of the Jacobian of the flow
$J_{t\leftarrow0}^{\mathbf{X}}(y_{0})$. Cass, Lyons~\cite{CL} and
Inahama~\cite{I} establish such integrability for the Brownian rough
path, but
only by using the independence of the increments; for more general Gaussian
processes a more careful analysis is needed. To understand the
difficulty of
this problem, we note from~\cite{FV} that the standard deterministic estimate
on $J_{t\leftarrow0}^{\mathbf{X}}(y_{0})$ gives
%
\begin{equation}
\bigl\llvert J_{t\leftarrow0}^{\mathbf{x}}(y_{0})\bigr\rrvert \leq C
\exp \bigl( C\Vert  \mathbf{x}\Vert
_{p\mbox{-}\operatorname{var};
[ 0,T ] }^{p} \bigr).\label{standard}%
\end{equation}
The case where $\mathbf{X}$ is a Gaussian rough path and $p>2$ (i.e.,
Brownian-type paths or rougher) the Fernique-type estimates of~\cite{FO10}
unfortunately only give that $\Vert \mathbf{X}\Vert_{p\mbox{-}\operatorname{var}; [ 0,T ] }$ has a Gaussian tail. The
right-hand side of (\ref{standard}) is hence not integrable in general. Worse
still, the work Oberhauser and Friz~\cite{FO09} shows that the inequality
(\ref{standard}) can actually be saturated for a (deterministic) choice of
differential equation and driving rough path. The essential
contribution of
this paper is that for random processes having enough structure (in particular
for Gaussian processes) only a set of small (or zero) measure comes
close to
equality in (\ref{standard}). What is therefore needed (and what we provide!)
is a deterministic estimate which respects the fine structure of path, and
which allows us to more strongly interrogate its probabilistic structure.

Our results will allow us to deduce the existence of moments of all
orders for
$J_{t\leftarrow0}^{\mathbf{x}}(y_{0})$ for RDEs driven by a class of Gaussian
processes (including, but not restricted to, fBm with Hurst index
$H>1/4)$. In
fact, our main estimate shows much more than simple moment estimates. Namely,
that the logarithm of the Jacobian has a tail that decays faster than an
exponential. To be a little more precise, we will show that
%
\begin{equation}
P \bigl( \log \bigl[ \bigl\llvert J_{\cdot\leftarrow0}^{\mathbf{X}} (
y_{0} ) \bigr\rrvert _{p\mbox{-}\operatorname{var}; [ 0,T ] } \bigr] >x \bigr) \lesssim\exp
\bigl( -x^{r} \bigr) \label{tail}%
\end{equation}
for any $r<r_{0}\in(1,2]$. The constant $r_{0}$ will be described in
terms of
the regularity properties of the Gaussian path.

The results are relevant to a number of important problems. First, they are
necessary if one wants to extend the work of~\cite{HM} and
\cite{HP} on the ergodicity of non-Markovian systems. Second, they are
an important ingredient in a Malliavin calculus proof on the smoothness
of the
density for RDEs driven by rough Gaussian noise in the elliptic setting.
Furthermore, they allow one to achieve an analogue of H\"{o}rmander's theorem
for Gaussian RDEs in conjunction with a
suitable version of Norris's lemma; see~\cite{N,Nu06}. In this
context, we remark that Hu and Tindel~\cite{HT} have recently obtained a
Norris lemma for fBm with $H>1/3$ and proved smoothness-of-density
results for
a class of nilpotent RDEs. Hairer and Pillai~\cite{H3} have also proved
H\"{o}rmander-type theorems for a general class of RDEs; their results are
predicated on the assumption that the Jacobian has finite moments of all
order. Hence, one application of this paper is to use the tail estimate
(\ref{tail}) together with the results in~\cite{HT} or~\cite{H3} to conclude
that for $t>0$ the law of $Y_{t}$ [the solution to (\ref{Int-Eq})] will,
under H\"{o}rmander's condition, have a smooth density w.r.t. Lebesgue measure
on $\mathbb{R}
^{e}$, for a rich classes of Gaussian processes $X$ which includes fBm
$H>1/3$.
All of these problems (and many more besides) require the existence of
high-order moments of the Malliavin covariance matrix of $Y_{t} (
\omega ),$ which is itself expressed in terms of  the Jacobian.

The techniques developed in this paper are relevant to the study of more
general RDEs, and not just the one solved by the Jacobian. Our estimates
can be applied to any random variable that can be controlled in
terms of $N_{\alpha,p,I} ( \mathbf{\cdot} ) $, which is a ``greedy''
approximation of the local $p$-variation we will introduce later. Similar
deterministic estimates we derive can also be obtained in  the following
cases (cf. Friz, Victoir~\cite{FV}):

\begin{longlist}[(1)]
\item[(1)] RDEs driven along linear vector fields of the form
$V_{i}(z)=A_{i}%
z+b_{i}$ for $e\times e$ matrices $A_{i}$ and $b_{i}$ in $\mathbb{R}
^{e};$

\item[(2)] higher order derivatives of the flow (subject to suitably enhance
regularity on the vector fields defining the flow);

\item[(3)] the inverse of the Jacobian of the flow;

\item[(4)] situations where one wants to control the distance between
two RDE
solutions in the (inhomogeneous) rough path metric (e.g., in fixed
point theorems).
\end{longlist}

Recent work~\cite{GL} has extended the class of linear-growth RDEs for which
we have nonexplosion and there may be scope to extend our results to this
setting. In this paper we focus only on the Jacobian because of its central
role in the wide range of problems we have outlined and obtain explicit bounds
for the tails of the distribution of the Jacobian.

We now outline the structure of the paper. In Section~\ref{sec2} we introduce some
important notation and concepts on the theory of rough paths. Because
this is
now standard and there are many references available (e.g.,~\cite{L,LCL,FV,LQ}), we keep the detail to a minimum. In Section
\ref{relation-m-n} we derive a quantitative bound on the growth of
$J_{t\leftarrow0}^{\mathbf{x}}$; the estimates we derive here are based very
closely on~\cite{FV}. We end up with a control on $J_{t\leftarrow
0}^{\mathbf{x}}$ in terms of a function on the space on (rough) path space
which we (suggestively) name the accumulated $\alpha$ local $p$-variation
[denoted by $M_{\alpha,I,p} ( \cdot ) $]. When $\mathbf{X}$ is
taken to be a Gaussian rough path the integrability properties of
$M_{\alpha,I,p} ( \mathbf{X} ) $ are not immediately obvious or
easy to study. We therefore spend time in Section~\ref{relation-m-n} deriving
a relationship between $M_{\alpha,I,p} ( \cdot ) $ and another
function on path space, which we denote $N_{\alpha,p,I} ( \mathbf
{\cdot
} ) $. The analysis at this stage remains entirely deterministic.
Section~\ref{sectGRP} records some facts about Gaussian rough paths,
including the crucial embedding theorems for Cameron--Martin spaces
that have
been derived in~\cite{FV}. We then present the main tail estimate on
$N_{\alpha,p,I} ( \mathbf{X} ) $. Our analysis is based on Gaussian
isoperimetry and more specifically Borell's inequality, which we
recall. Once
this is achieved we can use the relationship between $J_{t\leftarrow
0}^{\mathbf{X}}$ and $N_{\alpha,p,I} ( \mathbf{X} ) $ to exhibit
the stated tail behavior of $J_{t\leftarrow0}^{\mathbf{X}}$. This estimate
then constitutes our main result.

\section{Rough path concepts and notation}\label{sec2}

There are now many articles and texts providing an overview on rough path
theory (e.g.,~\cite{LCL} and~\cite{FV}, to name just two). We will focus
on establishing the notation we need for the current application. We will
study continuous $\mathbb{R}^{d}$-valued paths $x$ parameterized by time on a compact interval $I$
(sometimes $I$ will be taken to be $ [ 0,T ] $), and we
denote the
space of such functions by $C ( I,
\mathbb{R}
^{d} )$. We write $x_{s,t}:=x_{t}-x_{s}$ as a shorthand for the
increments of a path when $x$ in $C ( I,
\mathbb{R}
^{d} )$. For $p\geq$ $1$ we will use%
\[
\llvert x\rrvert _{\infty}:=\sup_{t\in I}\llvert
x_{t}\rrvert,\qquad \llvert x\rrvert _{p\mbox{-}\operatorname{var};I}:=
\biggl( \sup_{D [ I ] = ( t_{j} ) }\sum_{j:t_{j}\in D [
I ] }\llvert
x_{t_{j},t_{j+1}}\rrvert ^{p} \biggr) ^{1/p},
\]
and we refer to these quantities both symbolically and by name (they are,
resp., the uniform norm and the $p$-variation semi-norm). We denote by
$C^{p\mbox{-}\operatorname{var}} ( I,
\mathbb{R}
^{d} ) $ the linear subspace of $C ( I,
\mathbb{R}
^{d} ) $ consisting of path of finite $p$-variation. In the case where
$x$ is in $C^{p\mbox{-}\operatorname{var}} ( I,%
\mathbb{R}
^{d} ) $ and $p$ is in $[1,2),$ the iterated integrals of x are
canonically defined by Young integration. The collection of all these iterated
integrals together then gives the signature: for $s<t$ in $I$%
\[
S ( x ) _{s,t}:=1+\sum_{k=1}^{\infty}
\int_{s<t_{1}<t_{2}%
<\cdots<t_{k}<t}\,dx_{t_{1}}\otimes \,dx_{t_{2}}\otimes\cdots
\otimes \,dx_{t_{k}}\in T \bigl(
\mathbb{R}
^{d} \bigr).
\]
By writing $S ( x ) _{\inf I,\cdot}$ we can regard the
signature as
a path (on $I)$ with values in the tensor algebra. In a similar way, the
truncated signature%
\[
S_{N} ( x ) _{s,t}:=1+\sum_{k=1}^{N}
\int_{s<t_{1}<t_{2}%
<\cdots<t_{k}<t}\,dx_{t_{1}}\otimes \,dx_{t_{2}}\otimes\cdots
\otimes \,dx_{t_{k}}\in T^{N} \bigl(
\mathbb{R}
^{d} \bigr)
\]
is a path in the truncated tensor algebra, $T^{N} (
\mathbb{R}
^{d} ) $. It is a well-known fact that the path $S_{N} (
x )
_{\inf I,\cdot}$ takes values in the step-$N$ free nilpotent group with $d$
generators, which we denote $G^{N} (
\mathbb{R}
^{d} ) $. More generally, if $p\geq1$ we can consider the set of such
group-valued paths%
\[
\mathbf{x}_{t}= \bigl( 1,\mathbf{x}_{t}^{1},\ldots,
\mathbf{x}_{t}^{\lfloor
p\rfloor} \bigr) \in G^{\lfloor p\rfloor} \bigl(
\mathbb{R}
^{d} \bigr).
\]
The advantage this offers is that the group structure provides a natural
notion of increment, namely $\mathbf{x}_{s,t}:=\mathbf{x}_{s}^{-1}%
\otimes\mathbf{x}_{t}$. We can describe the set of ``norms'' on
$G^{\lfloor
p\rfloor} (
\mathbb{R}
^{d} ) $ which are homogeneous with respect to the natural scaling
operation on the tensor algebra; see~\cite{FV} for definitions and details.
The subset of these so-called homogeneous norms which are symmetric and
sub-additive~\cite{FV} gives rise to genuine metrics on $G^{\lfloor
p\rfloor} (
\mathbb{R}
^{d} ),$ which in turn gives rise to a notion of homogenous
$p$-variation metrics $d_{p\mbox{-}\operatorname{var}}$ on the $G^{\lfloor p\rfloor
} (
\mathbb{R}
^{d} ) $-valued paths. Let
%
\begin{equation}
\Vert  \mathbf{x}\Vert  _{p\mbox{-}\operatorname{var}; [
0,T ] }= \Biggl(
\sum_{i=1}^{ \lfloor p \rfloor}\sup_{D= (
t_{j} ) }
\sum_{j:t_{j}\in D}\bigl\llvert \mathbf{x}_{t_{j},t_{j+1}}%
^{i}\bigr\rrvert _{ (
\mathbb{R}
^{d} ) ^{\otimes i}}^{p/i} \Biggr)
^{1/p}\label{homogeneousnorm},%
\end{equation}
and note that if (\ref{homogeneousnorm}) is finite, then $\omega (
s,t ):=\Vert \mathbf{x}\Vert
_{p\mbox{-}\operatorname{var}; [ s,t ] }^{p}$ is a control (i.e., it is a
continuous, nonnegative, super-additive function on the simplex $\Delta
_{T}= \{  ( s,t ) \dvtx0\leq s\leq t\leq T \} $ which
vanishes on the diagonal.)

The space of weakly geometric $p$-rough paths [denoted $WG\Omega
_{p} (
\mathbb{R}
^{d} ) $] is the set of paths parameterised over $I$ although this is
often implicit with values in $G^{\lfloor p\rfloor} (
\mathbb{R}
^{d} ) $ such that (\ref{homogeneousnorm}) is finite. A refinement of
this notion is the space of geometric $p$-rough paths, denoted $G\Omega
_{p} (
\mathbb{R}
^{d} ) $, which is the closure of
\[
\bigl\{ S_{ \lfloor p \rfloor} ( x ) _{\inf I,\cdot
}\dvtx x\in C^{1\mbox{-}\operatorname{var}}
\bigl( I,%
\mathbb{R}
^{d} \bigr) \bigr\}
\]
with respect to the rough path metric $d_{p\mbox{-}\operatorname{var}}$.

We will often end up considering an RDE driven by a path $\mathbf{x}$ in
$WG\Omega_{p} (
\mathbb{R}
^{d} ) $ along a collection of vector fields $V= ( V^{1},\ldots,V^{d} ) $ on $\mathbb{R}
^{e}$. And from the point of view of existence and uniqueness results, the
appropriate way to measure the regularity of the $V_{i}$s results turns
out to
be the notion of Lipschitz-$\gamma$ (short: $\Lip\mbox{-}\gamma$) in the sense of
Stein\setcounter{footnote}{3}\footnote{See~\cite{FV} and~\cite{LCL}, and note the contrast with
classical Lipschitzness.}. This notion provides a norm on the space of such
vector fields (the $\Lip\mbox{-}\gamma$ norm), which we denote $\llvert
\cdot\rrvert _{\Lip\mbox{-}\gamma},$ and for the collection of vector
fields $V$
we will often make use of the shorthand
\[
\llvert V\rrvert _{\Lip\mbox{-}\gamma}=\max_{i=1,\ldots,d}\llvert
V_{i}\rrvert _{\Lip\mbox{-}\gamma},
\]
and refer to the quantity $\llvert  V\rrvert _{\Lip\mbox{-}\gamma}$ as the
$\Lip\mbox{-}\gamma$ norm of $V$.

\section{Translated rough paths}\label{translate}

Suppose $\mathbf{x}= ( 1,\mathbf{x}^{1},\ldots,\mathbf{x}^{ \lfloor
p \rfloor} ) $ is a weakly geometric $p$-rough path. If $h$
is in
$C^{q\mbox{-}\operatorname{var}} ( I,
\mathbb{R}
^{d} ) $ and $1/p+1/q>1,$ then the cross-iterated integrals
between $h$
and $\mathbf{x}$ exists canonically\vadjust{\goodbreak} using Young integration. This gives rise
to the so-called translated rough path $T_{h}\mathbf{x}$. The definition,
which is standard, can be found, for example, in~\cite{FV} or \cite
{LQ}. A~key
technical estimate used in the paper will involve this object. Before
we state
and prove this estimate, we recall the specific structure of
$T_{h}\mathbf{x}$
at the first two nontrivial tensor levels. For levels one and two we
have%
\begin{eqnarray*}
( T_{h}\mathbf{x} ) ^{1} & =&
\mathbf{x}^{1}+h,
\\
( T_{h}\mathbf{x} ) ^{2} & =&
\mathbf{x}^{2}+\int h\otimes \,d\mathbf{x}^{1}+\int
\mathbf{x}^{1}\otimes \,dh+\int h\otimes \,dh.
\end{eqnarray*}
The higher order terms become increasingly tiresome to write down. We
will not
go beyond the levelt $(
\mathbb{R}
^{d} ) ^{\otimes3},$ so we simply record for reference that this
can be
written as%
%
\begin{eqnarray}\label{thirdlevel}
( T_{h}\mathbf{x} ) _{s,t}^{3} & =&
\mathbf{x}_{s,t}^{3}+\int%
_{s}^{t}
\int_{s}^{v}h_{s,u}\otimes
\,dh_{u}\otimes \,dh_{v}
\nonumber
\\
&&{} +\int_{s}^{t}\mathbf{x}^{2}{}_{s,u}
\otimes \,dh_{u}+\int_{s}^{t}\int
_{s} 
^{v}\mathbf{x}_{s,u}^{1}
\otimes \,dh_{u}\otimes \,d\mathbf{x}_{v}^{1}-\int
_{s}%
^{t}h_{s,u}\otimes d
\mathbf{x}^{2}{}_{u,t}
\nonumber
\\[-8pt]
\\[-8pt]
\nonumber
&&{} +\int_{s}^{t}\int_{s}^{v}h_{s,u}
\otimes \,d\mathbf{x}_{u}^{1}\otimes \,dh_{v}+\int
_{s}^{t}\int_{s}^{v}
\mathbf{x}^{1}{}_{s,u}\otimes \,dh_{u}\otimes
\,dh_{v}\\
&&{}+\int_{s}^{t}\int
_{s}^{v}h_{s,u}\otimes
\,dh_{u}\otimes \,d\mathbf{x}%
_{v}^{1}.
\nonumber
\end{eqnarray}

 The proof of the following result will occupy the remainder of this
section. The lemma is important. It explains how we can control
the $p$-variation of the translated rough path by the sum of the
$p$-variation of the
(untranslated) rough path and the $q$-variation of the path by which we translate.

\begin{lemma}
\label{translationestimates}Let $1\leq p<4$. Suppose that $\mathbf{x}$
is a
weakly geometric $p$-rough path parametrised over a compact interval
$I$. Let
$h$ be a path in $C^{q\mbox{-}\operatorname{var}} ( I,
\mathbb{R}
^{d} ) $ where $1/q+1/p>1$. If $T_{h}\mathbf{x}$ denotes the translated
rough path, then for any $ [ s,t ] \subseteq I$ we have the
estimate
\[
\Vert  T_{h}\mathbf{x}\Vert
_{p\mbox{-}\operatorname{var};%
[ s,t ] }^{p}\leq C_{p,q} \bigl[ \Vert
\mathbf{x}\Vert  _{p\mbox{-}\operatorname{var}; [ s,t ] }%
^{p}+\llvert
h\rrvert _{q\mbox{-}\operatorname{var}; [ s,t ]
}^{p} \bigr].
\]
The constant $C_{p,q}$ is given explicitly by
\[
C_{p,q}=2^{p-1}\bigl[1+c_{p,q}^{p/2}+c_{p/2,q}^{p/3}+c_{p,q}^{2p/3}
\bigr],
\]
where $c_{l,m}=2\cdot4^{1/l+1/m}\zeta ( \frac{1}{l}+\frac
{1}{m} ), $
and $\zeta$ is the classical Riemann zeta function.
\end{lemma}

\begin{pf}
We will only prove the lemma for the most difficult case $p\in\lbrack
3,4)$. By
definition we have that
\[
\Vert  T_{h}\mathbf{x}\Vert
_{p\mbox{-}\operatorname{var};%
[ s,t ] }^{p}=\bigl\llvert ( T_{h}\mathbf{x} )
^{1}\bigr\rrvert _{p\mbox{-}\operatorname{var}; [ s,t ] }^{p}+\bigl\llvert (
T_{h}\mathbf{x} ) ^{2}\bigr\rrvert _{p/2\mbox{-}\operatorname{var}; [ s,t ]
}^{p/2}+
\bigl\llvert ( T_{h}\mathbf{x} ) ^{3}\bigr\rrvert
_{p/3\mbox{-}\operatorname{var}; [ s,t ] }^{p/3},
\]
where for $i=1,2,3$ we have
\[
\bigl\llvert ( T_{h}\mathbf{x} ) ^{i}\bigr\rrvert
_{p/i\mbox{-}\operatorname{var};%
[ s,t ] }^{p/i}=\sup_{D [ s,t ] = ( t_{i} )
}\sum
_{i:t_{i}\in D [ s,t ] }\bigl\llvert ( T_{h}%
\mathbf{x} )
_{t_{i},t_{i+1}}^{i}\bigr\rrvert _{ (
\mathbb{R}
^{d} ) ^{\otimes i}}^{p/i}.
\]
Note that the formula for the translated rough path gives at level one
of the
tensor algebra%
%
\begin{eqnarray}\label{1}
\bigl\llvert ( T_{h}\mathbf{x} ) ^{1}\bigr\rrvert
_{p\mbox{-}\operatorname{var};
[ s,t ] }^{p} & \leq& \bigl[ \bigl\llvert \mathbf{x}^{1}
\bigr\rrvert _{p\mbox{-}\operatorname{var}; [ s,t ] }+\llvert h\rrvert _{q\mbox{-}\operatorname{var};%
[ s,t ] } \bigr]
^{p}
\nonumber
\\
& \leq&2^{p-1} \bigl[ \Vert  \mathbf{x}\Vert  _{p\mbox{-}\operatorname{var}; [ s,t ] }^{p}+\llvert h\rrvert _{q\mbox{-}\operatorname{var}; [ s,t ] }^{p}
\bigr]
\\
& =:&C_{1} ( p,q ) \bigl[ \Vert  \mathbf{x}%
\Vert  _{p\mbox{-}\operatorname{var}; [ s,t ] }^{p}+\llvert h\rrvert
_{q\mbox{-}\operatorname{var}; [ s,t ] }^{p} \bigr].\nonumber %
\end{eqnarray}
At level two we need to analyze%
%
\begin{eqnarray}\label{secondlevel}%
&&\sum_{i:t_{i}\in D [ s,t ] }  \biggl| \mathbf{x}_{t_{i},t_{i+1}}%
^{2}+\underbrace{\int_{t_{i}}^{t_{i+1}}h_{t_{i},u}
\otimes \,d\mathbf {x}_{u}^{1}%
}_{=:A_{t_{i},t_{i+1}}^{1}}+
\underbrace{\int_{t_{i}}^{t_{i+1}}\mathbf{x} 
_{t_{i},u}^{1}\otimes \,dh_{u}}_{=:A_{t_{i},t_{i+1}}^{2}}
\nonumber
\\[-8pt]
\\[-8pt]
\nonumber
&&\hspace*{135pt}\qquad{}+\int
_{t_{i}}^{t_{i+1}%
}h_{t_{i},u}\otimes
\,dh_{u}  \biggr| _{ (
\mathbb{R}
^{d} ) ^{\otimes2}}^{p/2}.
\end{eqnarray}
Using Young's inequality we have for $j=1,2$ that
\begin{eqnarray*}
\bigl\llvert A_{t_{i},t_{i+1}}^{j}\bigr\rrvert _{ (
\mathbb{R}
^{d} ) ^{\otimes2}}^{p/2}
& \leq& c_{p,q}^{p/2}\bigl\llvert \mathbf{x} 
^{1}\bigr\rrvert _{p\mbox{-}\operatorname{var}; [ t_{i},t_{i+1} ] }^{p/2}\llvert h\rrvert
_{q\mbox{-}\operatorname{var}; [ t_{i},t_{i+1} ] }^{p/2}
\\
& \leq&\frac{c_{p,q}^{p/2}}{2} \bigl( \bigl\llvert \mathbf{x}^{1}\bigr
\rrvert _{p\mbox{-}\operatorname{var}; [ t_{i},t_{i+1} ] }^{p}+\llvert h\rrvert _{q\mbox{-}\operatorname{var}; [ t_{i},t_{i+1} ] }^{p}
\bigr).
\end{eqnarray*}
And also
\[
\biggl\llvert \int_{t_{i}}^{t_{i+1}}h_{t_{i},u}
\otimes \,dh_{u}\biggr\rrvert _{ (
\mathbb{R}
^{d} ) ^{\otimes2}}\leq c_{p,q}
\llvert h\rrvert _{q\mbox{-}\operatorname{var};%
[ t_{i},t_{i+1} ] }^{2}.
\]
Hence, we can deduce that%
%
\begin{eqnarray}\label{crossterms}
&&\sup_{D [ s,t ] = ( t_{i} ) }\sum_{i:t_{i}\in D [
s,t ] }\bigl\llvert
A_{t_{i},t_{i+1}}^{j}\bigr\rrvert _{ (
\mathbb{R}
^{d} ) ^{\otimes2}}^{p/2}\nonumber\\
&&\qquad
\leq\frac{c_{p,q}^{p/2}}{2}\sup_{D [
s,t ] = ( t_{i} ) }\sum
_{i:t_{i}\in D [ s,t ]
} \bigl( \bigl\llvert \mathbf{x}^{1}\bigr
\rrvert _{p\mbox{-}\operatorname{var}; [
t_{i},t_{i+1} ] }^{p}+\llvert h\rrvert _{q\mbox{-}\operatorname{var}; [
t_{i},t_{i+1} ] }^{p}
\bigr)
\\
&&\qquad \leq\frac{c_{p,q}^{p/2}}{2} \bigl( \bigl\llvert \mathbf{x}^{1}\bigr
\rrvert _{p\mbox{-}\operatorname{var}; [ s,t ] }^{p}+\llvert h\rrvert _{q\mbox{-}\operatorname{var}; [ s,t ] }^{p}
\bigr),\nonumber%
\end{eqnarray}
and similarly%
%
\begin{equation}
\sup_{D [ s,t ] = ( t_{i} ) }\sum_{i:t_{i}\in D [
s,t ] }\biggl
\llvert \int_{t_{i}}^{t_{i+1}}h_{t_{i},u}\otimes
\,dh_{u}\biggr\rrvert _{ (
\mathbb{R}
^{d} ) ^{\otimes2}}^{p/2}\leq
c_{p,q}^{p/2}\llvert h\rrvert _{q\mbox{-}\operatorname{var}; [ t_{i},t_{i+1} ] }^{p}.
\label{h}%
\end{equation}
From (\ref{secondlevel}), (\ref{crossterms}) and (\ref{h}) we easily obtain
that
%
\begin{eqnarray}\label{2}
&& \bigl\llvert ( T_{h}\mathbf{x} ) ^{2}\bigr\rrvert
_{p/2\mbox{-}\operatorname{var};%
[ s,t ] }^{p/2}
\nonumber
\\
&&\qquad \leq4^{p/2-1} \bigl[ \bigl\llvert \mathbf{x}^{2}\bigr\rrvert
_{p/2\mbox{-}\operatorname{var};%
[ s,t ] }^{p/2}+c_{p,q}^{p/2} \bigl( \bigl
\llvert \mathbf{x}%
^{1}\bigr\rrvert _{p\mbox{-}\operatorname{var}; [ s,t ] }^{p}+
\llvert h\rrvert _{q\mbox{-}\operatorname{var}; [ s,t ] }^{p} \bigr)\nonumber\\
&&\hspace*{185pt}\qquad{} +c_{p,q}%
^{p/2}\llvert h\rrvert _{q\mbox{-}\operatorname{var}; [ s,t ]
}^{p} \bigr]
\\
&&\qquad \leq4^{ ( p-1 ) /2}c_{p,q}^{p/2} \bigl[ \Vert  \mathbf{x}\Vert _{p\mbox{-}\operatorname{var}; [ s,t ] }%
^{p}+
\llvert h\rrvert _{q\mbox{-}\operatorname{var}; [ s,t ]
}^{p} \bigr]
\nonumber\\
&&\qquad =:C_{2} ( p,q ) \bigl[ \Vert  \mathbf{x}%
\Vert  _{p\mbox{-}\operatorname{var}; [ s,t ] }^{p}+\llvert h\rrvert
_{q\mbox{-}\operatorname{var}; [ s,t ] }^{p} \bigr].\nonumber%
\end{eqnarray}

We finish the proof by performing a similar analysis on the third
level. We
need to bound $\llvert  ( T_{h}\mathbf{x} ) ^{3}\rrvert
_{p/3\mbox{-}\operatorname{var}; [ s,t ] }^{p/3}$. Recall that
%
\begin{eqnarray}\label{level3expanded}
( T_{h}\mathbf{x} ) _{s,t}^{3} & =&
\mathbf{x}_{s,t}^{3}+\int%
_{s}^{t}
\int_{s}^{v}h_{s,u}\otimes
\,dh_{u}\otimes \,dh_{v}
\nonumber
\\
&&{} +\underbrace{\int_{s}^{t}\mathbf{x}^{2}{}_{s,u}
\otimes \,dh_{u}}%
_{=:B_{s,t}^{1}}+\underbrace{\int
_{s}^{t}\int_{s}^{v}
\mathbf{x}_{s,u}%
^{1}\otimes \,dh_{u}
\otimes \,d\mathbf{x}_{v}^{1}}_{=:B_{s,t}^{2}}%
-
\underbrace{\int_{s}^{t}h_{s,u}\otimes d
\mathbf{x}^{2}{}_{u,t}}%
_{=:B_{s,t}^{3}}
\nonumber
\\[-8pt]
\\[-8pt]
\nonumber
&&{} +\underbrace{\int_{s}^{t}\int
_{s}^{v}h_{s,u}\otimes d
\mathbf{x}_{u}%
^{1}\otimes \,dh_{v}}_{=:C_{s,t}^{1}}+
\underbrace{\int_{s}^{t}\int_{s}%
^{v}\mathbf{x}^{1}{}_{s,u}\otimes
\,dh_{u}\otimes \,dh_{v}}_{=:C_{s,t}^{2}%
}\\
&&{}+\underbrace{\int
_{s}^{t}\int_{s}^{v}h_{s,u}
\otimes \,dh_{u}\otimes
\,d\mathbf{x}_{v}^{1}}_{=:C_{s,t}^{3}}.\nonumber
\end{eqnarray}
We can split this up by first looking at the ``pure'' terms
%
\begin{eqnarray}
&&\sum_{i:t_{i}\in D [ s,t ] }\bigl\llvert \mathbf {x}_{t_{i},t_{i+1}%
}^{3}
\bigr\rrvert _{ (
\mathbb{R}
^{d} ) ^{\otimes3}}^{p/3}  \leq\Vert
\mathbf{x}%
\Vert  _{p\mbox{-}\operatorname{var}; [ s,t ] }^{p},\label
{pure}
\\
\label{pureh}&&\sum_{i:t_{i}\in D [ s,t ] }\biggl\llvert \int_{t_{i}}^{t_{i+1}}%
\int_{t_{i}}^{v}h_{t_{i},u}\otimes
\,dh_{u}\otimes \,dh_{v}\biggr\rrvert _{ (
\mathbb{R}
^{d} ) ^{\otimes3}}^{p/3}\nonumber\\
& &\qquad\leq c_{p,q}^{2p/3}\sum_{i:t_{i}\in
D [ s,t ] }
\llvert h\rrvert _{q\mbox{-}\operatorname{var}; [
t_{t},t_{i+1} ] }^{p}
\\
&&\qquad \leq c_{p,q}^{2p/3}\llvert h\rrvert _{q\mbox{-}\operatorname{var}; [
s,t ] }^{p}.
\nonumber
\end{eqnarray}
Second, we analyze the mixed terms in (\ref{level3expanded}%
). The strategy here as before is to use Young's inequality.
For $j=1
$ or $j=3$ we have
%
\begin{eqnarray}\label{mixed1a}%
&&\sum_{i:t_{i}\in D [ s,t ] }\bigl\llvert B_{t_{i},t_{i+1}}%
^{j}\bigr\rrvert _{ (
\mathbb{R}
^{d} ) ^{\otimes3}}^{p/3}\nonumber\\
&&\qquad \leq
c_{p/2,q}^{p/3}\sum_{i:t_{i}\in
D [ s,t ] }\Vert  \mathbf{x}\Vert  _{p\mbox{-}\operatorname{var}; [ t_{i},t_{i+1} ] }^{2p/3}
\llvert h{}\rrvert _{q\mbox{-}\operatorname{var}; [ t_{i},t_{i+1} ] }^{p/3}
\nonumber
\\[-8pt]
\\[-8pt]
\nonumber
&&\qquad \leq c_{p/2,q}^{p/3}\sum_{i:t_{i}\in D [ s,t ] }
\biggl[ \frac
{2}{3}\Vert  \mathbf{x}\Vert
_{p\mbox{-}\operatorname{var}; [ t_{i},t_{i+1} ] }^{p}+\frac{1}{3}\llvert h{}\rrvert
_{q\mbox{-}\operatorname{var}; [ t_{i},t_{i+1} ] }^{p} \biggr]
\\
&&\qquad \leq\frac{2}{3}c_{p/2,q}^{p/3} \bigl[ \Vert  \mathbf{x}%
\Vert  _{p\mbox{-}\operatorname{var}; [ s,t ] }^{p}+
\llvert h{}\rrvert _{q\mbox{-}\operatorname{var}; [ s,t ] }^{p} \bigr].\nonumber
\end{eqnarray}
A similar calculation yields%
%
\begin{equation}
\sum_{i:t_{i}\in D [ s,t ] }\bigl\llvert B_{t_{i},t_{i+1}}%
^{2}\bigr\rrvert _{ (
\mathbb{R}
^{d} ) ^{\otimes3}}^{p/3}\leq
c_{p,q}^{2p/3} \biggl[ \frac{2}%
{3}\Vert  \mathbf{x}\Vert  _{p\mbox{-}\operatorname{var}; [ s,t ] }^{p}+
\frac{1}{3}\llvert h{}\rrvert _{q\mbox{-}\operatorname{var}; [ s,t ] }^{p} \biggr].
\label{mixed1b}%
\end{equation}
Finally we have for $j=1,2,3$
%
\begin{eqnarray}\label{mixed2}%
&&\sum_{i:t_{i}\in D [ s,t ] }\bigl\llvert C_{t_{i},t_{i+1}}%
^{j}\bigr\rrvert _{ (
\mathbb{R}
^{d} ) ^{\otimes3}}^{p/3}\nonumber\\
&&\qquad \leq
c_{p,q}^{2p/3}\sum_{i:t_{i}\in
D [ s,t ] }\bigl
\llvert \mathbf{x}^{1}\bigr\rrvert _{p\mbox{-}\operatorname{var}; [ t_{i},t_{i+1} ] }^{p/3}
\llvert h\rrvert _{q\mbox{-}\operatorname{var}; [ t_{i},t_{i+1} ] }^{2p/3}
\nonumber
\\[-8pt]
\\[-8pt]
\nonumber
& &\qquad\leq c_{p,q}^{2p/3}\sum_{i:t_{i}\in D [ s,t ] }
\biggl[ \frac
{1}{3}\Vert  \mathbf{x}\Vert
_{p\mbox{-}\operatorname{var}; [ t_{i},t_{i+1} ] }^{p}+\frac{2}{3}\llvert h{}\rrvert
_{q\mbox{-}\operatorname{var}; [ t_{i},t_{i+1} ] }^{p} \biggr]
\\
&&\qquad \leq c_{p,q}^{2p/3} \biggl[ \frac{1}{3}\Vert  \mathbf{x}%
\Vert  _{p\mbox{-}\operatorname{var}; [ s,t ] }^{p}+
\frac{2}%
{3}\llvert h{}\rrvert _{q\mbox{-}\operatorname{var}; [ s,t ]
}^{p}
\biggr].\nonumber
\end{eqnarray}
Using the fact that the estimates (\ref{pure}), (\ref{pureh}), (\ref
{mixed1a}), (\ref{mixed1b}) and (\ref{mixed2}) are uniform over all
partitions, we derive the the elementary bound
\[
\bigl\llvert ( T_{h}\mathbf{x} ) ^{3}\bigr\rrvert
_{p/3\mbox{-}\operatorname{var};
[ s,t ] }^{p/3}\leq8^{p/3-1} \bigl[ (\ref{pure})+(
\ref{pureh}%
)+(\ref{mixed1a})+ (~\ref{mixed1b} ) +(\ref{mixed2}) \bigr].
\]
This then yields
%
\begin{eqnarray}\label{3}%
&&\bigl\llvert ( T_{h}\mathbf{x} ) ^{3}\bigr\rrvert
_{p/3\mbox{-}\operatorname{var};%
[ s,t ] }^{p/3}\nonumber\\
&&\qquad \leq8^{p/3-1}\biggl[
\frac{4}{3}c_{p/2,q}^{p/3}%
+
\frac{10}{3}c_{p,q}^{2p/3}\biggr] \bigl[\Vert  \mathbf{x}%
\Vert  _{p\mbox{-}\operatorname{var}; [ s,t ] }^{p}+
\llvert h\rrvert _{q\mbox{-}\operatorname{var}; [ s,t ] }^{p} \bigr]
\nonumber
\\[-8pt]
\\[-8pt]
\nonumber
&&\qquad \leq8^{ ( p-1 ) /3} \bigl[ c_{p/2,q}^{p/3}+c_{p,q}^{2p/3}
\bigr] \bigl[ \Vert  \mathbf{x}\Vert
_{p\mbox{-}\operatorname{var}; [ s,t ] }^{p}+\llvert h\rrvert _{q\mbox{-}\operatorname{var}; [
s,t ] }^{p}
\bigr]
\\
&&\qquad =:C_{3} ( p,q ) \bigl[ \Vert  \mathbf{x}%
\Vert  _{p\mbox{-}\operatorname{var}; [ s,t ] }^{p}+\llvert h\rrvert
_{q\mbox{-}\operatorname{var}; [ s,t ] }^{p} \bigr].\nonumber
\end{eqnarray}

Putting together (\ref{1}), (\ref{2}) and (\ref{3}) we establish that
\[
\Vert  T_{h}\mathbf{x}\Vert
_{p\mbox{-}\operatorname{var};%
[ s,t ] }^{p}\leq\sum_{i=1,2,3} \bigl\{
C_{i} ( p,q ) \bigr\} \bigl[ \Vert  \mathbf{x}\Vert  _{p\mbox{-}\operatorname{var}; [ s,t ] }^{p}+\llvert h\rrvert _{q\mbox{-}\operatorname{var}; [ s,t ] }^{p}
\bigr],
\]
and since $\sum_{i=1,2,3} \{ C_{i} ( p,q )  \}
=2^{p-1}+4^{ ( p-1 ) /2}c_{p,q}^{p/2}+8^{ ( p-1 )
/3}[c_{p/2,q}^{p/3}+c_{p,q}^{2p/3}]$ the estimate follows.
\end{pf}

\section{Deterministic estimates for solutions to RDEs}\label{relation-m-n}

In this section we will develop the pathwise estimate obtained in the previous
section. To assist with the clarity of the presentation it will be important
to first introduce some definitions of the main objects featuring in
our discussion.

\begin{notation}
If $\mathbf{x}$ is a weakly geometric $p$-rough path, then we will let
$\omega_{\mathbf{x,}p}$ denote the control which is induced by $\mathbf
{x}$ in
the sense that
\[
\omega_{\mathbf{x,}p} ( s,t ) \equiv\Vert  \mathbf{x}\Vert  _{p\mbox{-}\operatorname{var}; [ s,t ] }^{p}.
\]
\end{notation}

\begin{definition}
Let $\alpha>0$ and $I\subseteq
\mathbb{R}$ be a compact interval. Suppose that $\omega\dvtx I\times I\rightarrow%
\mathbb{R}
^{+}$ is a control. We define the accumulated $\alpha$-local $\omega
$-variation by%
\[
M_{\alpha,I} ( \omega ) =\mathop{\sup_{D ( I )
= ( t_{i} ) }}_{\omega ( t_{i},t_{i+1} ) \leq\alpha}%
\sum_{i:t_{i}\in D ( I ) }\omega ( t_{i},t_{i+1}
).
\]
\end{definition}

\begin{remark}
\label{scaling}Note that we have the scaling property $\beta M_{\alpha
/\beta,I} ( \omega ) =  M_{\alpha,I} ( \beta\omega ) $ for
any $\beta>0$.
\end{remark}

Of special interest is the case when the control is induced (in the
sense of
the above notation) by a (weakly) geometric $p$-rough path.

\begin{definition}
Let $\alpha>0$ and $I\subseteq%
\mathbb{R}
$ be a compact interval. We define the accumulated $\alpha$-local
$p$-variation to be the nonnegative function $M_{\alpha,I,p}$ which
acts on
weakly geometric $p$-rough paths (parameterized over $I$) by
%
\begin{equation}
M_{\alpha,I,p} ( \mathbf{x} ) \equiv M_{\alpha,I} ( \omega_{\mathbf{x,}p}
).\label{a-loc-def}%
\end{equation}
\end{definition}

\begin{remark}
The function $M_{\alpha,I,p}$ is well-defined because the
super-additivity of the control $\omega_{\mathbf{x,}p}$ ensures that
\[
M_{\alpha,I,p} ( \mathbf{x} ) \leq\Vert  \mathbf{x}\Vert  _{p\mbox{-}\operatorname{var};I}^{p} <\infty
\]
for any weakly geometric rough path $\mathbf{x}$ (again, parameterized over
$I$).\break $M_{\alpha,I,p} ( \mathbf{x} ) $ is continuous and
increasing in $\alpha$, and it equals $\Vert \mathbf{x}%
\Vert _{p\mbox{-}\operatorname{var};I}^{p}$ whenever $\alpha\geq\Vert\mathbf{x}\Vert _{p\mbox{-}\operatorname{var};I}^{p} $.
\end{remark}

The following lemma shows how the $\alpha$-local $\omega$-variation can be
used to derive Lipschitz bounds on solutions to RDEs.\footnote{It should be
compared with Theorem 10.26 of~\cite{FV}, on which the proof is based.}

\begin{lemma}
Assume $\lfloor p\rfloor+1\geqslant\gamma>p\geq1$. Suppose that $\mathbf{x}$
is a weakly geometric $p$-rough path parameterized on $ [ 0,T
],$
and $V= ( V^{1},\ldots,V^{d} ) $ is a collection of $\Lip\mbox{-}\gamma$
vector fields on $\mathbb{R}
^{e}$. Let $ (U_{t\leftarrow0}^{\mathbf{x}} ( \cdot )
)_{t\in [ 0,T ] }$ denote the flow induced by the RDE%
\[
dy_{t}=V ( y_{t} ) \,d\mathbf{x}_{t},\qquad
y ( 0 ) =y_{0},
\]
so that $U_{\cdot\leftarrow0}^{\mathbf{x}} ( y_{0} ) \equiv
y$. If
$\omega$ is the control $\omega ( u,v ) \equiv|V|_{\Lip\mbox{-}\gamma}%
^{p}\omega_{\mathbf{x,}p} ( u,v ),$ then for any $y_{0}^{1}$ and
$y_{0}^{2}$ in $%
\mathbb{R}
^{e}$, any $\alpha>0$ and any $ [ s,t ] \subseteq [
0,T ], $ we have
%
\begin{eqnarray}\label{FVbase}
\qquad&& \bigl\llvert U_{\cdot\leftarrow0}^{\mathbf{x}} \bigl( y_{0}^{1}
\bigr) -U_{\cdot\leftarrow0}^{\mathbf{x}} \bigl( y_{0}^{2}
\bigr) \bigr\rrvert _{p\mbox{-}\operatorname{var}; [ s,t ] }
\nonumber
\\[-8pt]
\\[-8pt]
\nonumber
&&\qquad \leq C\llvert V\rrvert _{\Lip\mbox{-}\gamma}\Vert  \mathbf{x}\Vert  _{p\mbox{-}\operatorname{var}; [ s,t ] }\bigl\llvert y_{0}^{1}-y_{0}^{2}
\bigr\rrvert \exp \bigl[ C\max \bigl( 1,\alpha ^{-1} \bigr)
M_{\alpha, [ 0,T ] } ( \omega ) \bigr],
\end{eqnarray}
where $C$ is a constant depending only on $p$.
\end{lemma}

\begin{pf}
Let $y_{t}^{i}\equiv U_{t\leftarrow0}^{\mathbf{x}} ( y_{0}^{i}
) $
for $i=1,2$. We follow through the details in the proof of Theorem
10.26 of
\cite{FV}, with the exception that we enhance each application of their Lemma
10.63 by instead using Remark 10.64 of the same reference. The contents of
this remark can be improved so that we use an arbitrary truncation parameter
$\alpha$ rather than setting $\alpha=1$; the details are easily checked
and we
omit them. These calculations result in the following estimate:%
%
\begin{equation}\quad
\bigl\llvert y_{s,t}^{1}-y_{s,t}^{2}
\bigr\rrvert \leq C\bigl\llvert y_{0}^{1} 
-y_{0}^{2}\bigr\rrvert \omega ( s,t ) ^{1/p}\exp
\bigl[ C\max \bigl( 1,\alpha^{-1} \bigr) M_{\alpha, [ 0,T ] } ( \omega )
\bigr].\label{alphaone}%
\end{equation}
The estimate (\ref{FVbase}) then follows by an elementary computation.
\end{pf}

Using these Lipschitz estimates on the flow, it is a relatively simple matter
to derive growth bounds on the Jacobian. This is the content of the
following corollary.

\begin{corollary}
\label{Jacobianbound}Assume $\lfloor p\rfloor+1>\gamma>p\geq1$. Suppose that
$\mathbf{x}$ is a weakly geometric $p$-rough path parameterized on
$ [
0,T ],$ and $V= ( V^{1},\ldots,V^{d} ) $ is a collection of
$\Lip\mbox{-}\gamma$ vector fields on $%
\mathbb{R}
^{e}$. Let $ (U_{t\leftarrow0}^{\mathbf{x}} ( \cdot )
)_{t\in [ 0,T ] }$ denote the flow induced by the RDE%
\[
dy_{t}=V ( y_{t} ) \,d\mathbf{x}_{t},\qquad
y ( 0 ) =y_{0}.
\]
Then the derivative $J_{t\leftarrow0}^{\mathbf{x}} ( y_{0} ) $ of
$U_{t\leftarrow0}^{\mathbf{x}} ( y_{0} ) $ exists and
satisfies the
growth-bound%
%
\begin{eqnarray}\label{finalbound}
\qquad&&\bigl\llvert J_{\cdot\leftarrow0}^{\mathbf{x}} ( y_{0} ) \bigr\rrvert
_{p\mbox{-}\operatorname{var}; [ 0,T ] }
\nonumber
\\[-8pt]
\\[-8pt]
\nonumber
&&\qquad\leq C\llvert V\rrvert
_{\Lip\mbox{-}\gamma}\Vert
\mathbf{x}\Vert  _{p\mbox{-}\operatorname{var}; [ 0,T ] }\exp \bigl[ C\max \bigl( \llvert V
\rrvert _{\Lip\mbox{-}\gamma}^{p},\alpha^{-1} \bigr)
M_{\alpha, [ 0,T ],p} ( \mathbf{x} ) \bigr].
\end{eqnarray}
\end{corollary}

\begin{pf}
It is well known~\cite{FV} under these hypotheses that $U_{\cdot
\leftarrow0}^{\mathbf{x}} ( y_{0} ) $ is differentiable. Fix
$\alpha>0$ and define
\[
\delta:=\alpha\llvert V\rrvert _{\Lip\mbox{-}\gamma}^{p}.
\]
Let $h$ be in $
\mathbb{R}
^{e},$ and for a real number $\varepsilon$ let $y_{0}^{1}=y_{0}+\varepsilon
h$ and
$y_{0}^{2}=y_{0}$. Take $U_{\cdot\leftarrow0}^{\mathbf{x}} ( y_{0}%
^{i} ) \equiv y^{i}$ for $i=1,2$. Applying the previous lemma we obtain
that for any $ [ s,t ] \subseteq [ 0,T ] $
\[
\bigl\llvert y_{s,t}^{1}-y_{s,t}^{2}
\bigr\rrvert ^{p}\leq C^{p}\llvert V\rrvert
_{\Lip\mbox{-}\gamma}^{p}\Vert  \mathbf{x}\Vert
_{p\mbox{-}\operatorname{var}; [ s,t ] }^{p}\varepsilon^{p}\llvert h\rrvert
^{p}\exp \bigl[ Cp\max \bigl( 1,\delta^{-1} \bigr)
M_{\delta, [ 0,T ] } ( \omega ) \bigr].
\]

Dividing by $\varepsilon^{p}$, taking the limit as $\varepsilon\downarrow0$
and then
taking the supremum over all $\llvert  h\rrvert =1$ this estimate
becomes%
%
\begin{eqnarray}\label{Jacinc}%
\qquad&&\bigl\llvert J_{t\leftarrow0}^{\mathbf{x}} ( y_{0} )
-J_{s\leftarrow
0}^{\mathbf{x}} ( y_{0} ) \bigr\rrvert ^{p}
\nonumber
\\[-8pt]
\\[-8pt]
\nonumber
&&\qquad
\leq C^{p}\llvert V\rrvert _{\Lip\mbox{-}\gamma}^{p}\Vert  \mathbf{x}\Vert  _{p\mbox{-}\operatorname{var}; [ s,t ] }^{p}\exp \bigl[
Cp\max \bigl( 1,\delta^{-1} \bigr) M_{\delta, [ 0,T ] } ( \omega ) \bigr].
\end{eqnarray}
Fix an arbitrary partition $D$ of $ [ 0,T ] $. Then by summing the
terms in (\ref{Jacinc}) and using the super-additivity of $\omega
_{\mathbf{x,}p}$ it follows that
\begin{eqnarray*}
&& \biggl( \sum_{i:t_{i}\in D}\bigl\llvert
J_{t_{i+1}\leftarrow0}^{\mathbf
{x}%
} ( y_{0} ) -J_{t_{i}\leftarrow0}^{\mathbf{x}}
( y_{0} ) \bigr\rrvert ^{p} \biggr) ^{1/p}
\\
&&\qquad \leq C\llvert V\rrvert _{\Lip\mbox{-}\gamma}\Vert  \mathbf{x}\Vert t _{p\mbox{-}\operatorname{var}; [ 0,T ] }\exp \bigl[ C\max \bigl( 1,\delta^{-1}
\bigr) M_{\delta, [ 0,T ] } ( \omega ) \bigr].
\end{eqnarray*}
To finish the proof we first optimize over all partitions $D$ to give an
estimate on the $p$-variation. We then use the scaling property in Remark
\ref{scaling} and the definition of $M_{\alpha,I,p} ( \mathbf
{x} )
$ to obtain that%
\[
M_{\delta, [ 0,T ] } ( \omega ) =\frac{\delta}{\alpha
}M_{\alpha, [ 0,T ] } \biggl(
\frac{\alpha}{\delta}\omega \biggr) =\llvert V\rrvert _{\Lip\mbox{-}\gamma}^{p}M_{\alpha,I,p}
( \mathbf {x}%
).
\]
Putting everything together gives (\ref{finalbound}).
\end{pf}

We have succeeded in showing how the derivative of the flow can be controlled
by using the function $M_{\alpha,I,p} ( \mathbf{\cdot} ) $.
But it
is still not obvious how to get a handle on the tail behavior of
$M_{\alpha,I,p} ( \mathbf{\cdot} ) $ when we evaluate it at a
Gaussian $p$-rough path.  To expose the structure further, we will now
consider another function $N_{\alpha,I,p} ( \cdot ) $ on
$WG\Omega_{p} (
\mathbb{R}
^{d} ),$ which is closely related to $M_{\alpha,I,p} (
\mathbf{\cdot} ) $.  The following sequence will play an important
role in enabling us to achieve this.

\begin{definition}[(The greedy sequence)] Assume $\mathbf{x\in}$ $WG\Omega_{p} (
\mathbb{R}
^{d} ) $ is parameterized over a compact interval $I$. If $\alpha
>0$ we
define a nondecreasing sequence $ ( \tau_{i} ( \alpha,p,\mathbf{x} )  ) _{i=0}^{\infty}= ( \tau_{i} (
\alpha )  ) _{i=0}^{\infty}$ in $I$ in by
%
\begin{eqnarray}\label{stoppingtimes}
\tau_{0} ( \alpha ) & =&\inf I,
\nonumber
\\[-8pt]
\\[-8pt]
\nonumber
\tau_{i+1} ( \alpha ) & =&\inf \bigl\{ t\dvtx \Vert
\mathbf{x}\Vert  _{p\mbox{-}\operatorname{var}; [ \tau_{i},t ]
}^{p}\geq\alpha,
\tau_{i} ( \alpha ) <t\leq\sup I \bigr\} \wedge\sup I,
\nonumber
\end{eqnarray}
with the convention that $\inf\varnothing=+\infty$. We call this
sequence the
greedy sequence.
\end{definition}

\begin{remark}
Note that for $\tau_{i} ( \alpha ) <$ $\sup I$ and $\Vert \mathbf{x}\Vert _{p\mbox{-}\operatorname{var}; [ \tau
_{i} ( \alpha ),\sup I ] }^{p}\geq\alpha,$ $\tau
_{i+1} ( \alpha ) $ is intuitively the first time $\Vert \mathbf{x}\Vert _{p\mbox{-}\operatorname{var}; [ \tau
_{i} ( \alpha ),\cdot ] }^{p}$ reaches $\alpha$ (recall that
the \mbox{$p$-variation} is a continuous function).
\end{remark}

We want to show that the greedy sequence is actually a partition of~$I$; in
other words it has only a finite number of distinct terms which include the
endpoints. With this objective in mind we introduce the function
$N_{\alpha,I,p}\dvtx WG\Omega_{p} (
\mathbb{R}
^{d} ) \rightarrow%
\mathbb{R}
_{+}$ given by
%
\begin{equation}
N_{\alpha,I,p} ( \mathbf{x} ):=\sup \bigl\{ n\in%
\mathbb{N}
\cup \{ 0 \} \dvtx\tau_{n} ( \alpha
) <\sup I \bigr\}.\label{numberofjumps}%
\end{equation}
We note that $N_{\alpha,I,p}$ describes the size of the nontrivial part of
the sequence $ ( \tau_{i} ( \alpha )  ) _{i=0}^{\infty}$.
More precisely, the number of distinct terms in the sequence $ (
\tau_{i} ( \alpha )  ) _{i=0}^{\infty}$ equals $N_{\alpha,I,p} ( \mathbf{x} ) +1$. The partition of the interval given by
\[
\bigl\{ \tau_{i} ( \alpha ) \dvtx i=0,1,\ldots,N_{\alpha,I,p} (
\mathbf{x} ) +1 \bigr\}
\]
can now heuristically be thought of as a ``greedy'' approximation to the supremum in identity
(\ref{a-loc-def}), the definition of the accumulated
$\alpha$-local $p$-variation.

\begin{lemma}
For any $\alpha>0,$ $p\geq1$ and any compact interval $I$ the function
$N_{\alpha,I,p}\dvtx WG\Omega_{p} (
\mathbb{R}
^{d} ) \rightarrow%
\mathbb{R}
_{+}$ is well defined; that is, $N_{\alpha,I,p} ( \mathbf{x} )
<\infty$ whenever $\mathbf{x}$ is in $WG\Omega_{p} (
\mathbb{R}
^{d} ) $.
\end{lemma}

\begin{pf}
From the continuity of $\Vert \mathbf{x}\Vert _{p\mbox{-}\operatorname{var}; [ s,\cdot ] }$ we can deduce that%
\[
\Vert  \mathbf{x}\Vert  _{p\mbox{-}\operatorname{var}; [
\tau_{i-1} ( \alpha ),\tau_{i} ( \alpha )  ]
}^{p}=
\alpha\qquad\mbox{for }i=1,2,\ldots,N_{\alpha,I,p} ( \mathbf{x} ).
\]
Thus, the super-additivity of $\omega_{\mathbf{x,}p}$ implies that if
$\mathbf{x}$ is in $WG\Omega_{p} (
\mathbb{R}
^{d} ), $ then
\begin{eqnarray*}
\alpha N_{\alpha,I,p} ( \mathbf{x} ) &=&\sum_{i=1}^{N_{\alpha,I,p} ( \mathbf{x} ) }
\omega_{\mathbf{x},p} \bigl( \tau _{i-1} ( \alpha ),
\tau_{i} ( \alpha ) \bigr) \leq\omega_{\mathbf{x},p} \bigl( 0,
\tau_{N_{\alpha,I,p} ( \mathbf{x}%
) } ( \alpha ) \bigr) \\
&\leq&\Vert  \mathbf{x}\Vert  _{p\mbox{-}\operatorname{var}; [ 0,T ] }%
^{p}<\infty.
\end{eqnarray*}
\upqed\end{pf}

\begin{corollary}
\label{canonicalpartition}Let $\mathbf{x}$ be a path in $WG\Omega
_{p} (
\mathbb{R}
^{d} ) $ and suppose $\alpha>0$. Define the sequence $ ( \tau
_{i} ( \alpha )  ) _{i=0}^{\infty}$ by (\ref{stoppingtimes}),
and let $N_{\alpha,I,p} ( \mathbf{x} ) $ be given by
(\ref{numberofjumps}). Then the set
\[
D_{\tau}= \bigl\{ \tau_{i} ( \alpha ) \dvtx
i=0,1,\ldots,N_{\alpha,I,p} ( \mathbf{x} ) +1 \bigr\}
\]
is a partition of $I$.
\end{corollary}

\begin{pf}
This now follows immediately from the definition of $ ( \tau
_{i} (
\alpha )  ) _{i=0}^{\infty}$ and the fact that $N_{\alpha,I,p} ( \mathbf{x} ) $ is finite.\vadjust{\goodbreak}
\end{pf}

 The following proposition shows how we can use $N_{\alpha,I,p} (
\mathbf{x} ) $ to bound on the $\alpha$-local $p$-variation.

\begin{proposition}
$\label{NandM}$Let $p\geq1$ and suppose $\mathbf{x}$ is a path in
$WG\Omega_{p} (
\mathbb{R}
^{d} ) $ parameterized over the compact interval $I,$ and then for every
$\alpha>0$%
\[
M_{\alpha,I,p} ( \mathbf{x} ) \leq \bigl( 2N_{\alpha,I,p} ( \mathbf{x} ) +1
\bigr) \alpha.
\]
\end{proposition}

\begin{pf}
First note that the case $N_{\alpha,I,p} ( \mathbf{x} ) =0$
can be
dealt with trivially. We therefore can (and will) assume in the following
that $N_{\alpha,I,p} ( \mathbf{x} ) \geq1$. Let $D= \{
t_{i}\dvtx i=0,1,\ldots,n \} $ be any partition of $I$ with the
property that
%
\begin{equation}
\omega_{\mathbf{x},p} ( t_{i-1},t_{i} ) \leq\alpha\qquad\mbox{for all }i=1,\ldots,n.\label{assumptiononD}%
\end{equation}
Corollary~\ref{canonicalpartition} ensures that $D_{\tau}$ is a
partition of
$I$. We relabel the points in $D$ with reference to the partition
$D_{\tau}$
by writing $t_{i}=t_{j}^{l}$ for $i=1,2,\ldots,n,$ where $l$ indicates
which of
disjoint subintervals $ \{ (\tau_{i} ( \alpha ),\tau
_{i+1} ( \alpha ) ]\dvtx i=0,1,\ldots,N_{\alpha,I,p} ( \mathbf
{x}%
)  \} $ contains $t_{i},$ and $j$ orders the $t_{i}$s within each
of these subintervals. More precisely, $l\in \{ 0,1,\ldots,N_{\alpha,I,p} ( \mathbf{x} )  \} $ is the unique natural number such
that
\[
\tau_{l} ( \alpha ) <t_{i}\leq\tau_{l+1} ( \alpha
);
\]
and then $j\geq1$ is well defined by%
\[
j=i-\max_{t_{r}\leq\tau_{l} ( \alpha ) }r.
\]
For each $l\in \{ 0,1,\ldots,N_{\alpha,I,p} ( \mathbf{x} )
\} $ let $n_{l}$ denote the number of elements of $D$ in $(\tau
_{l} ( \alpha ),\tau_{l+1} ( \alpha ) ]$. Suppose now
for a contradiction that $n_{l}=0$. In this case, $t_{n_{l-1}}^{l-1}$ and
$t_{1}^{l+1}$ are two consecutive points of $D$ with $t_{n_{l-1}}^{l-1}%
\leq\tau_{l} ( \alpha ) <\tau_{l+1} ( \alpha )
<t_{1}^{l+1},$ and since the $ ( \tau_{i} ( \alpha )  )
_{i=0}^{\infty}$ are defined to be maximal [recall (\ref
{stoppingtimes})] we
have
\[
\omega_{\mathbf{x,}p} \bigl( t_{n_{l-1}}^{l-1},t_{1}^{l+1}
\bigr) >\omega_{\mathbf{x,}p} \bigl( \tau_{l} ( \alpha ),\tau
_{l+1} ( \alpha ) \bigr) =\alpha.
\]
This contradicts the assumptions on $D$ (\ref{assumptiononD}). We deduce
that $n_{l}\geq1$.

We observe that if $n_{l}\geq2$, then the super-additivity of $\omega
_{\mathbf{x,}p}$ results in
\[
\sum_{j=1}^{n_{l}-1}\omega_{\mathbf{x,}p}
\bigl( t_{j}^{l},t_{j+1}^{l} \bigr)
\leq\omega_{\mathbf{x,}p} \bigl( t_{1}^{l},t_{n_{l}}^{l}
\bigr) \qquad\mbox{for }l=0,1,\ldots,N_{\alpha,I,p} ( \mathbf{x} );
\]
thus, by a simple calculation we have%
%
\begin{eqnarray}\label{keyestimate}
&& \sum_{j=1}^{n}\omega_{\mathbf{x,}p} (
t_{j-1},t_{j} )
\nonumber
\\
&&\qquad \leq\sum_{l=0}^{N_{\alpha,I,p} ( \mathbf{x} ) -1} \bigl\{ \bigl[
\omega_{\mathbf{x,}p} \bigl( t_{n_{l}}^{l},t_{1}^{l+1}
\bigr) +\omega _{\mathbf{x,}p} \bigl( t_{1}^{l+1},t_{n_{l+1}}^{l+1}
\bigr) \bigr] 1_{ \{ n_{l+1}\geq2 \} }
\nonumber
\\[-8pt]
\\[-8pt]
\nonumber
&&\hspace*{134pt}\qquad{}+\omega_{\mathbf{x,}p} \bigl(
t_{n_{l}}%
^{l},t_{n_{l+1}}^{l+1}
\bigr) 1_{ \{ n_{l+1}=1 \} } \bigr\}
\\
&&\qquad\quad{} +\omega_{\mathbf{x,}p} \bigl( 0,t_{n_{0}}^{0} \bigr).\nonumber
\end{eqnarray}
To complete the proof we note that $\omega_{\mathbf{x,}p} ( t_{1}%
^{l+1},t_{n_{l+1}}^{l+1} ) \leq\alpha$ and $\omega_{\mathbf
{x,}p} (
0,t_{n_{0}}^{0} ) \leq\alpha$ by the definition of the sequence
$ (
t_{j}^{l} ) $. Furthermore we have $\omega_{\mathbf{x,}p} (
t_{n_{l}}^{l},t_{1}^{l+1} ) \leq\alpha$ because $t_{n_{l}}^{l}$ and
$t_{1}^{l+1}$ are two consecutive points in $D$. Hence, we may deduce from~(\ref{keyestimate}) that%
\[
\sum_{j=1}^{n}\omega_{\mathbf{x,}p} (
t_{j-1},t_{j} ) \leq \bigl( 2N_{\alpha,I,p} ( \mathbf{x} )
+1 \bigr) \alpha.
\]
Because the right-hand side of the last inequality does not depend on $D$,
optimizing over all such partitions gives the stated result.
\end{pf}

As a direct consequence of Proposition~\ref{NandM} and Corollary
\ref{Jacobianbound} we have the estimate
%
\begin{eqnarray}\label{Jac}%
\quad\qquad\bigl\llvert J_{\cdot\leftarrow0}^{\mathbf{x}} ( y_{0} ) \bigr\rrvert
_{p\mbox{-}\operatorname{var}; [ 0,T ] }&\leq &C\llvert V\rrvert
_{\Lip\mbox{-}\gamma}\Vert
\mathbf{x}\Vert  _{p\mbox{-}\operatorname{var}; [ 0,T ] }
\nonumber
\\[-8pt]
\\[-8pt]
\nonumber
&&{}\times{}\exp \bigl[ C\max \bigl( 1,\alpha
|V|_{\Lip\mbox{-}\gamma
}^{p} \bigr) \bigl( 2N_{\alpha, [ 0,T ],p} (
\mathbf{x}%
) +1 \bigr) \bigr].
\end{eqnarray}
If we take $\mathbf{x=X}$ to be Gaussian rough path, then the tail of the
Jacobian can be studied via the tail of $N_{\alpha,I,p} ( \mathbf{X}
) $. This will be the objective of the remainder of the paper.

\begin{remark}
There are several ways to obtain bounds for the Jacobian in terms of
$N_{\alpha,I,p} ( \mathbf{x} ) $. An alternative approach suggested
by the anonymous referee uses the Gronwall estimate
\[
\bigl\llvert J_{\cdot\leftarrow0}^{\mathbf{x}} ( y_{0} ) \bigr\rrvert
_{\infty; [ 0,T ] }\leq C\exp \bigl( C\llVert X\rrVert _{p\mbox{-}\operatorname{var}; [ 0,T ] }^{p}
\bigr).
\]
Using the cocycle property
\[
J_{t\leftarrow0}^{\mathbf{x}} ( y_{0} ) =J_{t\leftarrow
s}^{\mathbf{x}}
\bigl( U_{s\leftarrow0}^{\mathbf{x}} ( y_{0} ) \bigr)
J_{s\leftarrow0}^{\mathbf{x}} ( y_{0} )
\]
a simple induction argument gives the following bound on the infinity
norm of
the Jacobian:%
\[
\bigl\llvert J_{\cdot\leftarrow0}^{\mathbf{x}} ( y_{0} ) \bigr\rrvert
_{\infty; [ 0,T ] }\leq C\exp \bigl( C\alpha^{p}%
N_{\alpha,I,p} ( \mathbf{x} ) \bigr).
\]
This argument may be generalized to cover the $p$-variation of the Jacobian;
see, for example,~\cite{FR} where the authors implement a variant of
this idea based on
a previous version of this paper.
\end{remark}

\section{Gaussian rough paths}\label{sectGRP}

The previous section developed the pathwise estimates on $J_{t\leftarrow
0}^{\mathbf{x}} ( y_{0} ) $ we need. We learned that the
$p$-variation of $J_{t\leftarrow0}^{\mathbf{x}} ( y_{0} ) $
can be
bounded explicitly in terms of $N_{\alpha, [ 0,T ],p} (
\mathbf{x} ) $. The importance of controlling $J_{t\leftarrow
0}^{\mathbf{x}} ( y_{0} ) $ using $N_{\alpha, [ 0,T ],p} ( \mathbf{x} ) $, as opposed to simpler alternatives [see, e.g.,
identity (\ref{standard})], is best appreciated when the
driving rough path is taken to be random. Henceforth, we will distinguish
situations where the path is random by writing it in upper-case:\vadjust{\goodbreak}
$\mathbf{X}$.
Of special interest is when $\mathbf{X}$ is the lift\footnote{Recall
that by
$\mathbf{X}$ being a lift of $X$, we mean that the projection of
$\mathbf{X}$
to the first tensor level is exactly $X$.} of some continuous $
\mathbb{R}
^{d}$-valued Gaussian process $ ( X_{t} ) _{t\in I}$. A theory of
such Gaussian rough paths has been developed by a succession of authors
\cite{CQ,FV07,CFV,FO10}, and we will mostly work within
their framework.

To be more precise, we will assume that $X_{t}= (
X_{t}^{1},\ldots,X_{t}%
^{d} ) $ is a continuous, centered (i.e., mean zero) Gaussian
process with
independent and identically distributed components. Let $R\dvtx I\times
I\rightarrow%
\mathbb{R}$ denote the covariance function of any component, that is,
\[
R ( s,t ) =E \bigl[ X_{s}^{1}X_{t}^{1}
\bigr].
\]
Throughout we will assume that this process is realized on the abstract Wiener
space $ ( \mathcal{W},\mathcal{H},\mu ) $ where $\mathcal{W} $
$=C_{0} ( I,
\mathbb{R}
^{d} ) $, the space of continuous $
\mathbb{R}
^{d}$-valued functions on $I$. More precisely we mean that $X$ is the
canonical process on $\mathcal{W}$; that is, $X_{t} ( \omega )
=\omega ( t ),$ and $ ( X_{t} ) _{t\in I}$ has the
required Gaussian distribution under $\mu$. We recall the notion of the
``rectangular increments of $R$'' from~\cite{FVupdate}; these are
defined by%
\[
R\pmatrix{
s,t
\vspace*{2pt}\cr
u,v}:=E \bigl[ \bigl( X_{t}^{1}-X_{s}^{1}
\bigr) \bigl( X_{v}%
^{1}-X_{u}^{1}
\bigr) \bigr].
\]
The existence of a lift for $X$ is guaranteed by insisting on a sufficient
rate of decay on the correlation of the increments. This is captured,
in a
very general way, by the following two-dimensional $\rho$-variation constraint
on the covariance function.

\begin{condition}
\label{rhovar} There exists of $1\leq\rho<2$ such that $R$ has
\textit{finite }$\rho$-\textit{variation} in the sense%
%
\begin{equation}
V_{\rho} ( R;I\times I ):=\lleft( \mathop{ \sup
_{D= (
t_{i} ) \in\mathcal{D} ( I ) }}_{D^{\prime}= (
t_{j}^{\prime} ) \in\mathcal{D} ( I ) }\sum_{i,j}
\biggl| R \pmatrix{
t_{i},t_{i+1}
\vspace*{2pt}\cr
t_{j}^{\prime},t_{j+1}^{\prime}%
} \biggr| ^{\rho} \rright)
^{{1}/{\rho}}<\infty.\label{2dvariation}%
\end{equation}
\end{condition}

\begin{remark}
\label{FVlift} Under Theorem 35, Condition~\ref{rhovar} of
\cite{FV07}, $ ( X_{t} ) _{t\in [ 0,T ] }$
lifts to a geometric $p$-rough path for any $p>2\rho$. Moreover, there
is a
unique \textit{natural lift} which is the limit (in the $d_{p\mbox{-}\operatorname{var}}$-induced
rough path topology) of  the canonical lift of piecewise linear
approximations to $X$.
\end{remark}

The following theorem  appears in~\cite{FV07} as Proposition 17; cf.
also the
recent note~\cite{FVupdate}. It shows how the assumption $V_{\rho} (
R; [ 0,T ] ^{2} ) <\infty$ allows us to embed~$\mathcal{H}$
in the space of continuous paths with finite $\rho$ variation. The
result, as
it appears in~\cite{FV07}, applies to one-dimensional Gaussian
processes. The
generalization to arbitrary finite dimensions is straightforward, and
we will
not elaborate on the proof.

\begin{theorem}[(\cite{FV07})]\label{CMpVarembedding}Let $ ( X_{t} ) _{t\in
I}= ( X_{t}^{1},\ldots,X_{t}^{d} ) _{t\in I}$ be a continuous,
mean-zero Gaussian process with independent and identically distributed
components. Let $R$ denote the covariance function of  (any) one of the
components. Then if $R$ is of finite $\rho$-variation for some $\rho\in
\lbrack1,2)$ we can embed $\mathcal{H}$ in the space $C^{\rho\mbox{-}\operatorname{var}}
( I,
\mathbb{R}
^{d} );$ in fact,%
%
\begin{equation}
\llvert h\rrvert _{\mathcal{H}}\geq\frac{\llvert  h\rrvert
_{\rho\mbox{-}\operatorname{var};I}}{\sqrt{V_{\rho} ( R;I\times I ) }%
}.\label{embedding}%
\end{equation}

\end{theorem}

\begin{remark}[(\cite{FV3})]\label{fBMembedding} Writing $\mathcal{H}^{H}$ for the
Cameron--Martin space of fBm for~$H$ in $ ( 1/4,1/2 ) $, the
variation embedding in~\cite{FV3} gives the stronger result that%
\[
\mathcal{H}^{H}\hookrightarrow C^{q\mbox{-}\operatorname{var}} \bigl( I,
\mathbb{R}
^{d} \bigr)\qquad \mbox{for
any }q> ( H+1/2 ) ^{-1}.
\]
\end{remark}

Once we have established a lift $\mathbf{X}$ of $X$ we will often want
to make
sense of $\mathbf{X} ( \omega+h ) $. The main technique used for
achieving this is to relate it to the translated rough path
$T_{h}\mathbf{x}$; recall Section~\ref{translate}. The the following
result appeared in
\cite{CFV} and demonstrates that, under certain conditions, $\mathbf
{X} (
\omega+h ) $ and $T_{h}\mathbf{X} ( \omega ) $ are equal for
all $h$ in $\mathcal{H}$ on a set of $\mu$-full measure.

\begin{lemma}
\label{translation} Let $ ( X_{t} ) _{t\in I}= ( X_{t}%
^{1},\ldots,X_{t}^{d} ) _{t\in I}$ be a mean-zero Gaussian process with
i.i.d. components. Assume that $X$ has a natural lift to a geometric $p$-rough path. Assume further that for some $q\geq1$ such that
$1/p+1/q>1$, we
have $\mathcal{H}\hookrightarrow C^{q\mbox{-}\operatorname{var}} ( I,%
\mathbb{R}
^{d} ) $. Then there exists a measurable subset $E\subseteq\mathcal{W}$
with $\mu ( E ) =1$, such that for all $\omega$ in $E$, we have
\[
T_{h}\mathbf{X} ( \omega ) \equiv\mathbf{X} ( \omega +h )\qquad\mbox{for all }h\mbox{ in }\mathcal{H}.
\]
\end{lemma}

From the different choices of $p$ and $q$ with the properties that $X$
 lifts
path in $G\Omega_{p} (
\mathbb{R}
^{d} ) $ and $\mathcal{H}$ continuously embeds in $C^{q\mbox{-}\operatorname{var}%
} ( I,%
\mathbb{R}
^{d} ),$ it will often prove useful to work with a particular choice
that satisfies certain constraints. The purpose of the next lemma is to show
that these constraints can always be satisfied for some $p$ and $q,$
for the
examples of Gaussian processes that will interest us most.

\begin{corollary}
\label{conditions}Let $ ( X_{t} ) _{t\in I}= ( X_{t}%
^{1},\ldots,X_{t}^{d} ) _{t\in I}$ be a continuous, mean-zero Gaussian
process with i.i.d. components on $ ( \mathcal{W},\mathcal{H},\mu
)
$. Suppose that at least one of the following holds:

\begin{longlist}[(1)]
\item[(1)] For some $\rho$ in $[1,\frac{3}{2})$ the covariance function
of $X$ has
finite $\rho$-variation, in the sense of Condition~\ref{rhovar};

\item[(2)] $X$ is a fractional Brownian motion for $H$ in $ (
1/4,1/2 )$.
\end{longlist}

Then there exist real numbers $p,q$ such that the following statements are
true simultaneously:\vadjust{\goodbreak}

\begin{longlist}[(1)]
\item[(1)] $X$ has a natural lift to a geometric $p$-rough path;

\item[(2)] $\mathcal{H}\hookrightarrow C^{q\mbox{-}\operatorname{var}} ( I,
\mathbb{R}
^{d} ) $ where $1/p+1/q>1$.
\end{longlist}
\end{corollary}

\begin{pf}
If Condition~\ref{rhovar} is satisfied with $\rho\in\lbrack1,3/2)$, then
(taking $\frac{1}{0}:=\infty)$
\[
2\rho<3<\frac{\rho}{\rho-1}.
\]
If we therefore set $q=\rho$ and choose $p$ in $ ( 2q,3 ),$ Remark
\ref{FVlift} guarantees the existence of a natural lift for $X$. Furthermore,
Theorem~\ref{CMpVarembedding} ensures that $\mathcal{H}\hookrightarrow
C^{q\mbox{-}\operatorname{var}} ( I,
\mathbb{R}
^{d} ) $.

In the case where $X$ is fBm let $4>p>\frac{1}{H},$ and then Remark~\ref{FVlift}
guarantees that $X$ lifts to a geometric $p$-rough path. Let $q= (
\frac{1}{p}+\frac{1}{2} ) ^{-1}$. Then we have
\[
\biggl( H+\frac{1}{2} \biggr) ^{-1}<q
\]
and
\[
\frac{1}{p}+\frac{1}{q}=\frac{2}{p}+\frac{1}{2}>1.
\]
The fact that $\mathcal{H}\hookrightarrow C^{q\mbox{-}\operatorname{var}} ( I,
\mathbb{R}
^{d} ) $ now follows by Remark~\ref{fBMembedding}.
\end{pf}

\section{\texorpdfstring{The tail behavior of $N_{\alpha,I,p} (\mathbf{X} (\cdot))$ via Gaussian isoperimetry}
{The tail behavior of N alpha,I,p (X (.)) via Gaussian isoperimetry}}

We continue to work in the setting of an abstract Wiener space $ (
\mathcal{W},\mathcal{H},\mu ) $. If $\mathcal{K}$ denotes the unit ball
in $\mathcal{H}$, then for any $A\subseteq\mathcal{W}$ we can consider the
Minkowski sum%
\[
A+r\mathcal{K}:= \{ x+ry\dvtx x\in A,y\in\mathcal{K} \}.
\]
We then recall the following isoperimetric inequality of C. Borell; cf.
Theorem 4.3 of~\cite{Led}.

\begin{theorem}[(Borell)]Let $ ( \mathcal{W},\mathcal{H},\mu ) $ be an abstract
Wiener space and $\mathcal{K}$ denote the unit ball in $\mathcal{H}$. Suppose
$A$ is a Borel subset of $\mathcal{W}$ such that $\mu ( A )
\geq\Phi ( a ) $ for some real number $a$. Then for every
$r\geq0$,%
\[
\mu_{\ast} ( A+r\mathcal{K} ) \geq\Phi ( a+r ),
\]
where $\mu_{\ast}$ is the inner measure of $\mu,$ and $\Phi$ denotes the
standard normal cumulative distribution function.
\end{theorem}

The next proposition is crucial. It will allow us to apply Borell's inequality
to control the tail of the random variable $N_{\alpha,I,p} (
\mathbf{X} ( \omega )  ) $.

\begin{proposition}
\label{borellestimate}Let $ ( X_{t} ) _{t\in I}= ( X_{t}%
^{1},\ldots,X_{t}^{d} ) _{t\in I}$ be a continuous, mean-zero Gaussian
process, parameterized over a compact interval $I$ on the abstract Wiener
space $ ( \mathcal{W},\mathcal{H},\mu ) $. Suppose that $p$
and $q$
are real numbers such that $1\leq p<4$ and $1/p+1/q>1$. Assume further that:\vadjust{\goodbreak}

\begin{longlist}[(1)]
\item[(1)] $X$ has a natural lift to a geometric $p$-rough path $\mathbf{X;}$

\item[(2)] $\mathcal{H}\hookrightarrow C^{q\mbox{-}\operatorname{var}} ( I,
\mathbb{R}
^{d} ) $.
\end{longlist}

Then there exists a set $E\subseteq\mathcal{W}$, of $\mu$-full measure, with
the following property: for all $\omega$ in $E,h$ in $\mathcal{H}$ and
$\alpha>0,$ if
\[
\bigl\Vert \mathbf{X} ( \omega-h ) \bigr\Vert
_{p\mbox{-}\operatorname{var};I}\leq\alpha,
\]
then
\[
\llvert h\rrvert _{q\mbox{-}\operatorname{var};I}\geq \alpha
N_{\tilde{\alpha}^{p},I,p} \bigl( \mathbf{X} ( \omega ) \bigr) ^{1/q},
\]
where $C_{p,q}$ is the constant in Lemma~\ref{translationestimates} and
$\widetilde{\alpha}= ( 2C_{p,q} ) ^{1/p}\alpha$.
\end{proposition}

\begin{pf}
Fix $\alpha>0$. We first note that the case $N_{\tilde{\alpha}^{p},I,p} (
\mathbf{X} ( \omega )  ) =0$ is trivial. Hence we will assume
in the following that $N_{\tilde{\alpha}^{p},I,p} ( \mathbf{X} (
\omega )  ) \geq1$. From the definition of the sequence $ (
\tau_{i} ( \tilde{\alpha}^{p} )  ) _{i=0}^{\infty}$ and the
integer $N_{\tilde{\alpha}^{p},I,p} ( \mathbf{X} ( \omega )
) $ we have for $i=0,1,2,\ldots,N_{\tilde{\alpha}^{p},I,p} (
\mathbf{X} ( \omega )  ) -1$
%
\begin{equation}
\bigl\Vert \mathbf{X} ( \omega ) \bigr\Vert
_{p\mbox{-}\operatorname{var}; [ \tau_{i} ( \tilde{\alpha}^{p} ),\tau
_{i+1} ( \tilde{\alpha}^{p} )  ] }=\widetilde{\alpha }\label{stoppingtime}.%
\end{equation}
Consider the (measurable) subset of $\mathcal{W}$ defined by
\[
E:= \bigl\{ \omega\in\mathcal{W}\dvtx T_{h}\mathbf{X} ( \omega ) =
\mathbf{X} ( \omega+h )\ \forall h\in\mathcal{H} \bigr\},
\]
and recall from Lemma~\ref{translation} that $\mu ( E ) =1$. For
every $\omega$ in $E$ define a subset $F_{\alpha,\omega}$ of $\mathcal
{H}$ by%
\[
F_{\alpha,\omega}:= \bigl\{ h\in\mathcal{H}\dvtx\bigl\Vert
\mathbf{X} ( \omega-h ) \bigr\Vert _{p\mbox{-}\operatorname{var};
I}\leq\alpha \bigr
\}.
\]

Using the estimate in Lemma~\ref{translationestimates} we have for any
$\omega$ in $E$%
\begin{eqnarray*}
\bigl\Vert \mathbf{X} ( \omega )\bigr\Vert
_{p\mbox{-}\operatorname{var}; [ \tau_{i} ( \tilde{\alpha}^{p} ),\tau
_{i+1} ( \tilde{\alpha}^{p} )  ] }^{p} & =&\bigl\Vert T_{h}
\mathbf{X} ( \omega-h ) \bigr\Vert _{p\mbox{-}\operatorname{var}; [ \tau_{i} (\tilde{\alpha}^{p} ),\tau
_{i+1} ( \tilde{\alpha}^{p} )  ] }^{p}
\\
& \leq& C_{p,q} \bigl( \bigl\Vert T_{h}
\mathbf{X} ( \omega )\bigr\Vert _{p\mbox{-}\operatorname{var};I}^{p}+
\llvert h\rrvert _{q\mbox{-}\operatorname{var}; [ \tau_{i} ( \tilde{\alpha}^{p} ),\tau_{i+1} ( \tilde{\alpha}^{p} )  ] }^{p} \bigr).
\end{eqnarray*}
Hence, for any $\omega$ in $E$, $h$ in $F_{\alpha,\omega}$ we have%
%
\begin{equation}
\widetilde{\alpha}^{p}\leq C_{p,q} \bigl(
\alpha^{p}+\llvert h\rrvert _{q\mbox{-}\operatorname{var}; [ \tau_{i} ( \tilde{\alpha}^{p} ),\tau
_{i+1} ( \tilde{\alpha}^{p} )  ] }^{p} \bigr).
\label{est} 
\end{equation}
Substituting $ ( 2C_{p,q} ) ^{1/p}\alpha$ for $\widetilde
{\alpha}$, estimate (\ref{est}) becomes%
\[
\llvert h\rrvert _{q\mbox{-}\operatorname{var}; [ \tau_{i} (\tilde{\alpha}^{p} ),\tau_{i+1} ( \tilde{\alpha}^{p} )  ] }^{q}\geq
\alpha^{q}.
\]
Summing over $i=0,1,\ldots,N_{\widetilde{\alpha},I,p} ( \mathbf{X}%
( \omega )  ) -1$ then gives
\[
\llvert h\rrvert _{q\mbox{-}\operatorname{var};I}^{q}\geq\sum
_{i=0}%
^{N_{\tilde{\alpha}^{p},I,p} ( \mathbf{X} ( \omega )  )
-1}\llvert h\rrvert
_{q\mbox{-}\operatorname{var}; [ \tau_{i} (
\tilde{\alpha}^{p} ),\tau_{i+1} ( \tilde{\alpha}^{p} )
] }^{q}\geq
\alpha^{q}N_{\tilde{\alpha}^{p},I,p} \bigl( \mathbf{X} ( \omega ) \bigr),
\]
which yields the desired estimate.
\end{pf}

By using these estimates in concert with Borell's inequality
we are
lead directly to the following theorem which describes the needed tail-estimate
on the random variable $N_{\tilde{\alpha}^{p},I,p} (
\mathbf{X} ( \cdot )  ) $.

\begin{theorem}
\label{maintailestimate}Let $ ( X_{t} ) _{t\in I}= (
X_{t}^{1},\ldots,X_{t}^{d} ) _{t\in I}$ be a continuous, mean-zero Gaussian
process, parameterized over a compact interval $I,$ on the abstract Wiener
space $ ( \mathcal{W},\mathcal{H},\mu ) $. Suppose that $p$
and $q$
are real numbers satisfying $1\leq p<4$ and $1/p+1/q>1$. Assume that:

\begin{longlist}[(1)]
\item[(1)] $X$ has a natural lift to a geometric $p$-rough path $\mathbf{X;}$

\item[(2)] $\mathcal{H}\hookrightarrow C^{q\mbox{-}\operatorname{var}} ( I,
\mathbb{R}
^{d} ),$ so that there exists $C_{\mathrm{emb}}$ in $ ( 0,\infty
) $
with $\llvert  h\rrvert _{q\mbox{-}\operatorname{var};I}\leq C_{\mathrm{emb}}\llvert
h\rrvert _{\mathcal{H}}$ for all $h$ in $\mathcal{H}$.
\end{longlist}

Let $C_{p,q}$ be the constant in Lemma~\ref{translationestimates}. Then for
all $\alpha>0$ the natural lift $\mathbf{X}$ of $X$ to a geometric $p$-rough
path satisfies
%
\begin{equation}
\mu \bigl\{ \omega\dvtx N_{\tilde{\alpha}^{p},I,p} \bigl( \mathbf{X} ( \omega ) \bigr) >n
\bigr\} \leq C_{1}\exp \biggl[ \frac{-\alpha^{2}n^{2/q}}{2C_{\mathrm{emb}}^{2}} \biggr]
\label{requiredtailbound}%
\end{equation}
for all $n\geq1,$ where $\tilde{\alpha}= ( 2C_{p,q} ) ^{1/p}%
\alpha$. The constant $C_{1}$ is given explicitly by
%
\begin{equation}
C_{1}=\exp \bigl[ 2\Phi^{-1} \bigl( \mu ( A_{\alpha}
) \bigr) ^{2} \bigr],\label{event}%
\end{equation}
where $\Phi^{-1}$ is the inverse of the standard normal cumulative
distribution function and%
\[
A_{\alpha}:= \bigl\{ \omega\in\mathcal{W}\dvtx\bigl\Vert
\mathbf{X}%
( \omega )\bigr \Vert _{p\mbox{-}\operatorname{var};I}\leq
\alpha \bigr\}.
\]

\end{theorem}

\begin{pf}
By applying Proposition~\ref{borellestimate} together with hypothesis
2, we
can deduce that
%
\begin{equation}
\bigl\{ \omega\dvtx N_{\tilde{\alpha}^{p},I,p} \bigl( \mathbf{X} ( \omega ) \bigr) >n
\bigr\} \cap E\subset\mathcal{W\setminus} ( A_{\alpha}+r_{n}
\mathcal{K} ), \label{borell-prep}%
\end{equation}
where $E\subseteq\mathcal{W}$ with $\mu ( E ) =1$ and
\[
r_{n}:=\frac{\alpha n^{1/q}}{C_{\mathrm{emb}}}.
\]
Noticing that $\mu ( A_{\alpha} ) =:\Phi ( a_{\alpha} )
$ is in $ ( 0,1 ) $ [i.e., $a_{\alpha}$ is in $ (
-\infty,\infty ) $] an application of Borell's inequality then gives
that%
%
\begin{equation}
\mu \bigl\{ \omega\dvtx N_{\tilde{\alpha}^{p},I,p} \bigl( \mathbf{X} ( \omega ) \bigr) >n
\bigr\} \leq1-\Phi ( a_{\alpha}+r_{n} ) \leq\exp \biggl[ -
\frac{ ( a_{\alpha}+r_{n} ) ^{2}}{2} \biggr].\label{borellbound}%
\end{equation}
If $a_{\alpha}>-r_{n}/2$, then (\ref{borellbound}) implies%
\[
\mu \bigl\{ \omega\dvtx N_{\tilde{\alpha}^{p},I,p} \bigl( \mathbf{X} ( \omega ) \bigr) >n
\bigr\} \leq\exp \biggl( -\frac{r_{n}^{2}}%
{8} \biggr).
\]
Alternatively if $a_{\alpha}\leq-r_{n}/2$ then $r_{n}^{2}\leq4a_{\alpha}^{2},$
 and it is easy to see that%
\begin{eqnarray*}
\mu \bigl\{ \omega\dvtx N_{\tilde{\alpha}^{p},I,p} \bigl( \mathbf{X} ( \omega ) \bigr) >n
\bigr\} &\leq&\exp \biggl( -\frac{a_{\alpha}%
^{2}+2a_{\alpha}r_{n}}{2} \biggr) \exp \biggl( -
\frac{r_{n}^{2}}{2} \biggr)\\
& \leq&\exp \bigl( 2a_{\alpha}^{2}
\bigr) \exp \biggl( -\frac{r_{n}^{2}}%
{2} \biggr).
\end{eqnarray*}
Since $a_{\alpha}=$ $\Phi^{-1} ( \mu ( A_{\alpha} )
) $
we have shown the required estimate (\ref{requiredtailbound}).\vadjust{\goodbreak}
\end{pf}

\begin{remark}
Suppose that for some $\rho$ in $[1,\frac{3}{2})$ the covariance
function of~$X$ has finite $\rho$-variation (in the sense of Condition \ref
{rhovar}). In
this case we deduce from Corollary~\ref{conditions} and Theorem
\ref{CMpVarembedding} that $q=\rho$ and $p\in ( 2\rho,3 ) $
satisfy the hypothesis of Theorem~\ref{maintailestimate} with the embedding
constant given explicitly by
\[
C_{\mathrm{emb}}=\sqrt{V_{\rho} ( R;I\times I ) }.
\]
Hence, the tail estimates just proved lead to moment estimates on\break
$N_{\alpha,I,p} ( \mathbf{X} ( \omega )  ) $ in the
usual way. This leads to the conclusion that for any $\alpha>0,$ and
$\eta$
satisfying
\[
\eta<\frac{\alpha^{2}}{2V_{\rho} ( R;I\times I ) },
\]
we have
%
\begin{equation}
\int_{\mathcal{W}}\exp \bigl[ \eta N_{\tilde{\alpha}^{p},I,p} \bigl(
\mathbf{X} ( \omega ) \bigr) ^{2/\rho} \bigr] \mu ( d\omega ) <\infty.
\label{moment}%
\end{equation}
For the Brownian rough path (i.e., $\rho=1$) this shows that $N_{\alpha,I,p} ( \mathbf{X} ( \omega )  ) $ has a Gaussian tail
since in this case we have $\log\llvert  J_{t\leftarrow0}^{\mathbf
{X} (
\omega ) } ( y_{0} ) \rrvert \lesssim N_{\alpha,I,p} ( \mathbf{X} ( \omega )  ) $. Rudimentary It\^{o}
or Stratonovich calculus tells us that we cannot expect the tail of
$N_{\alpha,I,p} ( \mathbf{X} ( \omega )  ) $ to decay
any faster than Gaussian, suggesting a degree of sharpness to our approach.
By a similar argument, we can show that for any $r<2/\rho$
\[
\exp \bigl[ N_{\alpha,I,p} \bigl( \mathbf{X} ( \cdot ) \bigr) ^{r}
\bigr]\qquad \mbox{is in } \bigcap_{q>0}L^{q} ( \mu );
\]
and similar calculations can be performed in the fractional Brownian
setting too.
\end{remark}

\begin{theorem}[(Moment estimates on the Jacobian)]Let $ ( X_{t} ) _{t\in [
0,T ] }= ( X_{t}^{1},\ldots,X_{t}^{d} ) _{t\in [
0,T ] }$ be a continuous, mean-zero Gaussian process with i.i.d.
components associated to the abstract Wiener space $ ( \mathcal{W},\mathcal{H},\mu ) $. Let $\rho$ be in $[1,\frac{3}{2}),$ $p$ in
$ ( 2\rho,3 ) $ and $\gamma>p$. Suppose that the covariance
function of $X$ has finite $\rho$-variation in the sense of Condition
\ref{rhovar}. Then $X$ lifts to a geometric $p$-rough path $\mathbf
{X,}$ and
for any collection of $\Lip\mbox{-}\gamma$ vector fields $V= ( V^{1},\ldots,V^{d} ) $ on $
\mathbb{R}
^{e}$ with $\gamma>p$ the RDE%
\[
dY_{t}=V ( Y ) \,d\mathbf{X},\qquad Y ( 0 ) =y_{0}%
\]
has a unique solution. The flow $U_{t\leftarrow0}^{\mathbf{X} (
\omega ) } ( \cdot ) $ induced by the solution to this
RDE is
differentiable. Let this derivative be given by
\[
J_{t\leftarrow0}^{\mathbf{X} ( \omega ) } ( y_{0} ) \cdot a:= \biggl\{
\frac{d}{d\varepsilon}U_{t\leftarrow0}^{\mathbf{X} (
\omega ) } ( y_{0}+\varepsilon a
) \biggr\} _{\varepsilon
=0}.
\]
And let $M_{\mathbf{X} ( \cdot ) }^{ ( y_{0},V )
}\dvtx\mathcal{W\rightarrow
\mathbb{R}
}_{+}$ denote the random variable%
\[
M_{\mathbf{X} ( \cdot ) }^{ ( y_{0},V ) } ( \omega ) \equiv M_{\mathbf{X} ( \omega ) }^{ (
y_{0},V ) }:=
\bigl\llvert J_{t\leftarrow0}^{\mathbf{X} (
\omega ) } ( y_{0} ) \bigr\rrvert
_{p\mbox{-}\operatorname{var}; [
0,T ] }.
\]
Then for all $y_{0}$ in $
\mathbb{R}
^{e}$ and all $r<2/\rho$ we have
\[
\exp \bigl[ \bigl( \log M_{\mathbf{X} ( \cdot ) }^{ (
y_{0},V ) } \bigr) ^{r}
\bigr] \qquad\mbox{is in } \bigcap_{q>0}%
L^{q} ( \mu ).
\]
\end{theorem}

\begin{pf}
Remark~\ref{FVlift} guarantees the existence of a unique natural lift
$\mathbf{X}$ for~$X$. Furthermore, we know that if $V= ( V^{1},\ldots,V^{d} ) $ is any collection of $\Lip\mbox{-}\gamma$ vector fields (and
$\gamma>p),$ then the solution flow obtained by driving $\mathbf{X}$
along $V$
is differentiable. Lemma~\ref{Jacobianbound} and Proposition~\ref{NandM}
together yield (\ref{Jac}) from which it follows that for any $\alpha
>0$ and
$y_{0}$ in $
\mathbb{R}
^{e}$
\[
M_{\mathbf{X} ( \cdot ) }^{ ( y_{0},V ) }\leq c_{1}\bigl\Vert
\mathbf{X} ( \omega ) \bigr\Vert _{p\mbox{-}\operatorname{var}; [ 0,T ] }\exp \bigl[
c_{1}N_{\alpha,I,p} \bigl( \mathbf{X} ( \omega ) \bigr) \bigr],
\]
where $I= [ 0,T ] $ and $c_{1}$ is a nonrandom constant which
depends on $\alpha,p,\gamma$ and $\llvert  V\rrvert _{\Lip\mbox
{-}\gamma
}$. Without loss of generality we take $c_{1}>1$. Then for three further
(again nonrandom) constants $c_{2}$ and $c_{3}$ an easy calculation
gives%
\begin{eqnarray*}
\bigl( \log M_{\mathbf{X} ( \cdot ) }^{ ( y_{0},V )
} \bigr) ^{r} & \leq&
c_{2}+c_{3} \bigl( \log\bigl\Vert \mathbf{X}
( \omega ) \bigr\Vert _{p\mbox{-}\operatorname{var}; [
0,T ] } \bigr) ^{r}+c_{4}N_{\alpha,I,p}
\bigl( \mathbf{X} ( \omega ) \bigr) ^{r}
\\
& \leq &c_{5}+c_{3}\log\bigl\Vert
\mathbf{X} ( \omega ) \bigr\Vert _{p\mbox{-}\operatorname{var}; [ 0,T ] }^{r}+c_{4}N_{\alpha,I,p}
\bigl( \mathbf{X} ( \omega ) \bigr) ^{r}.%
\end{eqnarray*}
Hence, we have
%
\begin{equation}
\exp \bigl[ \bigl( \log M_{\mathbf{X} ( \cdot ) }^{ (
y_{0},V ) } \bigr) ^{r}
\bigr] \leq c_{5}\bigl\Vert \mathbf{X} ( \omega ) \bigr\Vert_{p\mbox{-}\operatorname{var}; [
0,T ] }^{c_{3}r}\exp \bigl[
c_{4}N_{\alpha,I,p} \bigl( \mathbf {X} ( \omega ) \bigr)
^{r} \bigr].\label{relation}%
\end{equation}
By Theorem~\ref{maintailestimate} and the remark following it, the random
variable%
\[
\exp \bigl[ c_{5}N_{\alpha,I,p} \bigl( \mathbf{X} ( \omega )
\bigr) ^{r} \bigr]
\]
on the right-hand side of (\ref{relation}) is $L^{q} ( \mu ) $ for
all $q>0$ provided $r<2/\rho$.  On the other hand $\Vert
\mathbf{X} ( \omega )\Vert _{p\mbox{-}\operatorname{var}; [
0,T ] }$ has a Gaussian tail (see~\cite{FV}), and hence also has finite
moments of all order.  Using these two observations together with
Cauchy--Schwarz inequality in (\ref{relation}) gives the desired conclusion.
\end{pf}

The above result applies (in particular) to fractional Brownian motion,
$H>1/3$. But in the case of fBm we can leverage the specific embedding
properties to state an alternative version of the theorem which applies when
$H>1/4$.

\begin{theorem}[(Fractional Brownian motion)]Let $ ( X_{t} ) _{t\in [
0,T ] }= ( X_{t}^{1},\ldots,\break  X_{t}^{d} ) _{t\in [
0,T ] }$ be fractional Brownian motion with Hurst parameter
$H>1/4$. Let
$ ( \mathcal{W},\mathcal{H},\mu ) $ denote the abstract associated
with $X$. Let $\gamma>p>1/H$. Then~$X$ lifts to a geometric $p$-rough path
$\mathbf{X}$, and if\vadjust{\goodbreak} $V= ( V^{1},\ldots,V^{d} ) $ is a collection of
$\Lip\mbox{-}\gamma$ vector fields on $
\mathbb{R}^{e}$, the RDE%
\[
dY_{t}=V ( Y ) \,d\mathbf{X},\qquad Y ( 0 ) =y_{0}%
\]
has a unique solution. The flow $U_{t\leftarrow0}^{\mathbf{X} (
\omega ) } ( \cdot ) $ induced by the solution to this
RDE is
differentiable. Let this derivative be given by
\[
J_{t\leftarrow0}^{\mathbf{X} ( \omega ) } ( y_{0} ) \cdot a:= \biggl\{
\frac{d}{d\varepsilon}U_{t\leftarrow0}^{\mathbf{X} (
\omega ) } ( y_{0}+\varepsilon a
) \biggr\} _{\varepsilon
=0},
\]
and let $M_{\mathbf{X} ( \cdot ) }^{ ( y_{0},V )
}\dvtx\mathcal{W\rightarrow
\mathbb{R}
}_{+}$ denote the random variable%
\[
M_{\mathbf{X} ( \cdot ) }^{ ( y_{0},V ) } ( \omega ) \equiv M_{\mathbf{X} ( \omega ) }^{ (
y_{0},V ) }:=
\bigl\llvert J_{t\leftarrow0}^{\mathbf{X} (
\omega ) } ( y_{0} ) \bigr\rrvert
_{p\mbox{-}\operatorname{var}; [
0,T ] }.
\]

Then for any $r<2H+1$, we have that
\[
\exp \bigl[ \bigl( \log M_{\mathbf{X} ( \cdot ) }^{ (
y_{0},V ) } \bigr) ^{r}
\bigr] \mbox{ is in} \bigcap_{q>0}%
L^{q} ( \mu )
\]
for all $y_{0}$ in $
\mathbb{R}
^{e}$.
\end{theorem}

\begin{pf}
The argument is the similar to that of the last theorem; we have to
verify the
hypothesis of Theorem~\ref{maintailestimate}. Notice first that if
$r_{1}<r_{2}$, then a simple calculation gives
\[
\exp \bigl[ \bigl( \log M_{\mathbf{X} ( \cdot ) }^{ (
y_{0},V ) } \bigr) ^{r_{1}}
\bigr] \leq1+\exp \bigl[ \bigl( \log M_{\mathbf{X} ( \cdot ) }^{ ( y_{0},V ) } \bigr)
^{r_{2}} \bigr].
\]
It is therefore sufficient to prove the result for $1<r<2H+1$. Fix such any
such~$r$. The fact that $X$  lifts to a geometric $p$-rough path for any
$p>H^{-1}$ is by now a familiar one. Since for any $\tilde{p}>p$ we have
\[
\bigl\llvert J_{t\leftarrow0}^{\mathbf{X} ( \omega ) } ( y_{0} ) \bigr\rrvert
_{\tilde{p}\mbox{-}\operatorname{var}; [ 0,T ] }%
\leq\bigl\llvert J_{t\leftarrow0}^{\mathbf{X} ( \omega ) } (
y_{0} ) \bigr\rrvert _{p\mbox{-}\operatorname{var}; [ 0,T ] },
\]
it is sufficient to prove the result for any $p$ satisfying%
\[
H^{-1}<p<\max \bigl[ 4,2 ( r-1 ) ^{-1} \bigr].
\]
Fix any such $p$ in this interval and let $q$ be given by
\[
q= \biggl( \frac{1}{p}+\frac{1}{2} \biggr) ^{-1}.
\]
It follows that
\[
r<\frac{2}{q}<1+2H.
\]
The calculations of Corollary~\ref{conditions} then ensure that
$\mathcal{H}%
\hookrightarrow C^{q\mbox{-}\operatorname{var}} (  [ 0,T ],
\mathbb{R}
^{d} ) $. This allows us to apply Theorem~\ref{maintailestimate} to
deduce that
%
\begin{equation}
\exp \bigl[ N_{\alpha,I,p} \bigl( \mathbf{X} ( \omega ) \bigr) ^{r}
\bigr] \label{nbound1}%
\end{equation}
is $\mu$-integrable. The result then follows by repeating the steps of the
proof of the previous theorem.\vadjust{\goodbreak}
\end{pf}

\begin{remark}
In particular these results imply (under the stated conditions) that
$\llvert  J_{t\leftarrow0}^{\mathbf{X} ( \omega ) } (
y_{0} ) \rrvert _{p\mbox{-}\operatorname{var}; [ 0,T ] }$ has finite
moments of all order.
\end{remark}

\section*{Acknowledgements}
We would like to thank the anonymous referee for a number of
suggestions that helped to improve the presentation of this
paper significantly.

%


\printaddresses

\end{document}